\documentclass[11pt,reqno]{amsart}
\usepackage{amsmath,amsthm,amssymb,isomath,mathrsfs}
\usepackage[leqno]{amsmath}
\usepackage{derivative}
\usepackage[new]{old-arrows}
\usepackage{graphicx}
\usepackage{float}
\usepackage[english]{babel}
\usepackage{latexsym}
\usepackage[utf8]{inputenc}
\usepackage{babel}
\usepackage{enumerate}
\usepackage{epsfig}
\usepackage{relsize}
\usepackage{layout}
\usepackage{mathtools}
\usepackage[bottom]{footmisc}
\usepackage[
margin=2cm,
includefoot,
footskip=30pt,
]{geometry}
\usepackage{layout}
\usepackage{tikz}
\usepackage{tikz-cd}
\usetikzlibrary{cd}
\usepackage{derivative}
\usepackage[all]{xy}
\usepackage{comment}
\usepackage{fancyhdr}
\usepackage{xcolor,hyperref}
\hypersetup{colorlinks=true, breaklinks,
	urlcolor=magenta, linkcolor = purple,
	anchorcolor=purple, citecolor=blue}
\usepackage{csquotes}
\usepackage{xcolor,hyperref}
\hypersetup{colorlinks=true, breaklinks,
	urlcolor=magenta, linkcolor = red,
	anchorcolor=purple, citecolor=blue}

\usepackage{libertine}
\usepackage{libertinust1math}
\usepackage[T1]{fontenc}

\newcommand{\C}{\mathbf{C}}
\renewcommand{\P}{\mathbf{P}}
\newcommand{\R}{\mathbf{R}}
\renewcommand{\phi}{\varphi}
\renewcommand{\L}{\mathcal{L}}
\DeclareMathAlphabet{\altmathcal}{OMS}{cmsy}{m}{n}
\renewcommand{\O}{\altmathcal{O}}

\newcommand{\I}{\mathcal{I}}
\newcommand{\F}{\mathcal{F}}

\newcommand{\E}{\mathcal{E}}
\newcommand{\Z}{\mathbf{Z}}
\newcommand{\bX}{\overline{X}}

\newcommand{\Sym}{\mathrm{Sym}}
\newcommand{\wX}{\widetilde{X}}
\newcommand{\wB}{\widetilde{B}}

\newcommand{\Pic}{\mathrm{Pic}}
\newcommand{\G}{\Gamma}
\renewcommand{\l}{\ell}

\renewcommand{\b}{\overline}
\newcommand{\w}{\widetilde}
\newcommand{\OPn}{\altmathcal{O}_{\mathbf{P}^{n}}}
\newcommand{\OPN}{\altmathcal{O}_{\mathbf{P}^{N}}}
\renewcommand{\o}{\otimes}

\newcommand{\eps}{\varepsilon}
\newcommand{\mult}{\mathrm{mult}}
\newcommand{\m}{\mathfrak{m}}
\newcommand{\J}{\mathcal{J}}
\newcommand{\Bs}{\mathrm{Bs}}

\newcommand{\s}{\sigma}

\renewcommand{\to}{\rightarrow}
\newcommand{\Ram}{\mathrm{Ram}}

\newcommand{\supp}{\mathrm{supp}}
\newcommand{\B}{\mathbf{B}}
\newcommand{\mc}{\mathcal}

\DeclareMathOperator{\rk}{rk}

\DeclareMathOperator{\Imm}{Im}
\DeclareMathOperator{\Spec}{Spec}

\theoremstyle{plain}
\newtheorem{thm}{Theorem}[section]
\theoremstyle{plain}
\newtheorem{bthm}{Theorem}
\theoremstyle{plain}

\theoremstyle{definition}
\newtheorem{defi}[thm]{Definition}
\theoremstyle{plain}
\newtheorem{prop}[thm]{Proposition}
\theoremstyle{plain}

\theoremstyle{plain}
\newtheorem{bcor}[bthm]{Corollary}
\theoremstyle{plain}
\newtheorem{lemma}[thm]{Lemma}
\theoremstyle{definition}
\newtheorem{rmk}[thm]{Remark}
\theoremstyle{plain}

\theoremstyle{plain}

\theoremstyle{definition}
\newtheorem{ex}[thm]{Example}
\theoremstyle{definition}
\newtheorem{notat}[thm]{Notation}

\numberwithin{equation}{section}

\title[Positivity of Ulrich bundles in the ample and free case]{Positivity of Ulrich bundles in the ample and free case}
\author[Valerio Buttinelli]{Valerio Buttinelli*}
\address{Dipartimento di Matematica "Guido Castelnuovo"\\
	Sapienza Universit\`a di Roma\\
	Piazzale Aldo Moro 5, 00185 Roma, Italy}
\thanks{*Work produced as a part of the author's PhD thesis under the supervision of Angelo Felice Lopez}
\email[Valerio~Buttinelli]{valerio.buttinelli@uniroma1.it}

\begin{document}

	\begin{abstract}
		We study the positivity of an Ulrich vector bundle defined with respect to a globally generated ample line bundle. First we prove a generalization of a Lopez theorem on the first Chern class and the bigness of an Ulrich bundle. Then, under some additional assumptions on the polarization, we give a description of its augmented base locus, which consequently leads to a characterization of the V-bigness and of the ampleness of an Ulrich bundle in this setting.
	\end{abstract}
	
	\maketitle
	
	\section{Introduction}
	A class of vector bundles which aroused a certain interest in recent years is the one of Ulrich bundles. Given a smooth projective variety $X\subset \P^N,$ a vector bundle $\E$ on $X$ is an Ulrich bundle if $H^i(X,\E(-p))=0$ for all $i\geq0$ and $1\leq p\leq \dim X.$ Despite such bundles were introduced in the framework of commutative algebra in \cite{ulrich1984gorenstein}, they attracted the attention of algebraic geometers after the papers \cite{eisenbud2003resultants, eisenbud2011boij}, where the authors showed that the Chow form of a variety $X$ carrying an Ulrich bundle is determinantal and that the (rational) cone of cohomology tables of vector bundles on $X$ is the same as that of $\P^{\dim X}.$ Since then, a lot of work has been done in the study of Ulrich bundles (see \cite{coskun2017survey, beauville2018introduction, costa2021ulrich} for an account). However, the conjecture raised by Eisenbud and Schreyer in \cite{eisenbud2003resultants} concerning the existence of such bundles is still wide open, even in dimension two.
	
	In this paper, we will consider a slightly more general setting: the polarization will no longer be very ample, but simply ample and globally generated.
	
	\begin{defi}
		Let $X$ be a smooth projective variety of dimension $n\geq 1$ and let $B$ a globally generated ample line bundle. A vector bundle $\E$ is said to be \emph{$B$-Ulrich} if $H^i(X,\E(-pB))=0$ for all $i\geq0$ and $1\leq p\leq n.$
	\end{defi}
	
	This choice has three principal motivations:  first of all, the basic properties of the “usual Ulrich bundles” continue to hold also in this situation (Lemma \ref{lem:claim}). Secondly, many (smooth projective) varieties, such as Del Pezzo manifolds of degree $2$ or smooth cyclic coverings, come with a natural base-point-free polarization which generally is not very ample but that reveals more features on the underlying variety than a possible embedding. In this direction, it has been proved in \cite{sebastian2022rank} that a double cover $\phi\colon S\to \P^2$ branched over a general smooth irreducible curve of degree $2s\geq6$ carries a special Ulrich bundle of rank $2$ for $B=\phi^\ast\O_{\P^2}(1).$ In any case, if a variety supports an Ulrich bundle with respect to a globally generated polarization, then there is an Ulrich bundle for all very ample multiples (Lemma \ref{lem:claim}(vi)). Finally, the theory obtained by relaxing the hypothesis on the polarization is different from that of Ulrich bundles defined with respect to a hyperplane section: in Examples \ref{ex:DP-surface} - \ref{ex:DP-threefold} we find two Ulrich bundles defined with respect to a globally generated (but non-very ample) polarization which are not big and which are not ascribable to the classification of non-big (usual) Ulrich bundles on surfaces and threefolds in \cite[Theorems 1 - 2]{lopez2021classification}.
	
	In this work, we will study the positivity of a $B$-Ulrich bundle defined with respect to a globally generated ample line bundle $B.$ The first result concerns the first Chern class and the bigness of the Ulrich bundle. The goal is to prove an analogue of \cite[Theorem 7.2 \& Theorem 1]{lopez2022positivity} trying to follow the same arguments. In this way one gets aware of the differences with the usual setting. The first issue that we meet is that we generally lose the separation properties of a hyperplane section. It becomes necessary to study the local positivity of the polarization in order to find the necessary condition on the variety to recover those properties. To this end, we will consider the Seshadri constants of $B$ and, to get a more geometric insight, we will also prove the existence of some Seshadri curves for $B$ (Lemmas \ref{lem:7.1} - \ref{lem:7.1-2}). At the end we will show that if the variety is not (generically) covered by such curves, then the $c_1$ is positive and the Ulrich bundle is big. In particular, we get the following result.
	
	\begin{bthm}\label{thm:main}
		Let $X$ be a smooth projective variety of dimension $n\geq1$ and let $B$ a globally generated ample line bundle on $X$ with $B^n=d.$ Let $\E$ be a vector bundle of rank $r$ which is $0$-regular with respect to $B.$ Then 
		\begin{equation}\label{eq:ulrich-c1}
			c_1(\E)^k\cdot Z\geq r^k\mult_x(Z)
		\end{equation} holds for every $x\in X$ and for every subvariety $Z\subset X$ of dimension $k\geq1$ passing through $x$ provided that the following conditions are satisfied:
		\begin{itemize}\setlength\itemsep{0em}
			\item[(a)] $x\notin\Ram(\phi_B),$
			\item[(b)] $\eps(B;\phi_B^{-1}(\phi_{B}^{\hspace{0.000001mm}}(x)))>1.$
		\end{itemize} 
		In particular, if $X$ is not generically covered by $1$-Seshadri curves for $\phi_B,$
		then $\E$ is V-big and
		\begin{equation}\label{eq:ulrich-bigness}
			c_1(\E)^n\geq r^n.
		\end{equation}
		Moreover, if $\E$ is $B$-Ulrich of rank $r\geq 2,$ then
		\begin{equation}\label{eq:ulrich-c1^n}
			c_1(\E)^n\geq r(d-1).
		\end{equation}
	\end{bthm}
	
	The positivity of a vector bundle is strictly related to its asymptotic base loci (Definition \ref{def:augmented} and \cite[Theorem 1.1]{bauer2015positivity}). A $B$-Ulrich bundle is generated by global sections (Lemma \ref{lem:claim}(i)), hence the stable base locus and the restricted base locus are empty. Under some additional assumptions on the polarization $B,$ still including all very ample line bundles, we can obtain a complete description of the augmented base locus of a $B$-Ulrich bundle $\E$: it is the union of all $B$-lines, as defined below, on which $\E$ is not ample. If we weaken the hypotheses on $B,$ this characterization (and also Theorem \ref{thm:ulrich-ampleness}) may no longer hold (see Remarks \ref{rmk:DP-surface} - \ref{rmk:DP-threefold} - \ref{rmk:ulrich-curve}).
	
	\begin{defi}\label{def:B-line}
		Let $X$ be projective variety and let $B$ be a globally generated ample line bundle on $X.$ A curve $C\subset X$ is a \emph{$B$-line} if $B\cdot C=1.$
	\end{defi}
	
	To do this we will make use of the characterization of the augmented base locus of a vector bundle in terms of Seshadri constant (see Definition \ref{def:seshadri-vector} and Remark \ref{rmk:augmented-vector}). In conclusion we prove the following.
	
	\begin{bthm}\label{thm:ulrich-augmented}
		Let $X$ be a smooth projective variety and let $B$ a globally generated ample line bundle such that there is a linear series $|V|\subseteq |B|$ inducing a morphism which is étale onto its schematic image. Let $\E$ be a $B$-Ulrich bundle on $X$ and let $x\in X$ be a point. Then $\eps(\E;x)=0$ if and only if there exists a $B$-line $\G\subset X$ passing through $x$ such that $\E_{|\G}$ is not ample on $\G.$ 
		In particular, 
		\begin{equation}\label{eq:ulrich-augmented}
			\B_+(\E)=\bigcup_{\G\subset X}\G
		\end{equation} 
		where $\G$ ranges over all $B$-lines in $X$ such that $\E_{|\G}$ is not ample on $\G.$
	\end{bthm}
	
	A simple consequence of Theorem \ref{thm:ulrich-augmented} is the characterization of the V-bigness (Definition \ref{def:V-big}) and of the ampleness of an Ulrich bundle in this setting. 
	
	\begin{bcor}\label{cor:ulrich-augmented}
		Let $X$ be a smooth projective variety and let $B$ a globally generated ample line bundle such that there is a linear series $|V|\subseteq |B|$ inducing a morphism which is étale onto its schematic image. Let $\E$ be a $B$-Ulrich bundle. Then:
		\begin{itemize}
			\item[(a)]  $\E$ is V-big if and only if $X$ is not covered by $B$-lines $\G\subset X$ on which $\E_{|\G}$ is not ample.
			\item[(b)] $\E$ is ample if and only if either $X$ contains no $B$-lines or $\E_{|\G}$ is ample on every $B$-line $\G\subset X.$
		\end{itemize}
	\end{bcor}
	
	Regarding the ampleness of an Ulrich bundle defined with respect to a polarization as above, we can be more precise. All the technical results about the “separation properties” of the bundle, needed for Theorem \ref{thm:ulrich-augmented}, lead to a (very slight) generalization of \cite[Theorem 1]{lopez2023geometrical} in this setting.
	
	\begin{bthm}\label{thm:ulrich-ampleness}
		Let $X$ be a smooth projective variety and let $B$ a globally generated ample line bundle such that there is a linear series $|V|\subseteq |B|$ inducing a morphism which is étale onto its schematic image. Let $\E$ be a $B$-Ulrich bundle on $X.$ Then the following are equivalent:
		\begin{itemize}
			\item[(1)]$\E$ is $1$-very ample.
			\item[(2)] $\E$ is very ample.
			\item[(3)] $\E$ is ample.
			\item[(4)] Either $X$ contains no $B$-lines or $\E_{|\G}$ is ample on every $B$-line $\G\subset X.$
		\end{itemize}
	\end{bthm}
	
	\subsection*{Acknowledgments} I am deeply in debt with Angelo Felice Lopez for his continuous support. 
	
	\section{Notations}
	Throughout the paper we will adopt the following conventions:
	\begin{itemize}
		\setlength\itemsep{0em}
		\item All schemes are separated and of finite type  over the field of complex numbers $\C.$
		\item A variety is an integral scheme and subvarieties are always closed. 
		\item A point $x\in X$ is always closed unless otherwise specified and its ideal sheaf is denoted by $\m_x.$ The ideal sheaf of a closed subscheme $Y\subset X$ is denoted by $\I_{Y/X}.$
		\item Given a closed subscheme $Y\subset X,$ by \emph{“blow-up of $X$ along $Y$”} we will mean the blow-up of $X$ along the ideal sheaf $\I_{Y/X}.$
		\item Given a line bundle $L,$ the morphism induced by a linear series $|V|$ for a non-zero vector space $V\subset H^0(X,L)$ is denoted by $\phi_V\colon X\to\P^N.$ We refer to $\Bs(|V|)$ as the \emph{base locus} of $|V|$ or as the \emph{base scheme} of $|V|$ if we want to emphasize the scheme structure induced by the \emph{base ideal} $\mathfrak{b}(|V|)=\Imm(V\o L^\ast\to\O_X).$ 
		\item Given a line bundle $L,$ we write $\F(pL)=\F\o L^{\o p}$ for any sheaf $\F$ and any $p\in\Z.$
		\item A line bundle $L$ is \emph{strictly nef} if $L\cdot C>0$ for every irreducible curve $C.$
		\item The ramification locus of a finite morphism $f$ is denoted by $\Ram(f).$
	\end{itemize}
	
	\section{Preliminary definitions and generalities on Ulrich vector bundles}\label{sect:ulrich}
	In this section we collect some definitions and useful facts.
	
	\begin{defi}
		A vector bundle $\F$ on a projective variety $X$ is \emph{nef} (resp. \emph{big, ample, very ample}) if the line bundle $\O_{\P(\F)}(1)$ is nef (resp. big, ample, very ample) on the projectivized bundle $\P(\F).$ Given an integer $k\geq0,$ we say that $\F$ is \emph{$k$-very ample} if the restriction map $$H^0(X,\F)\to H^0(Z,\F_{|Z})$$ is surjective for every $0$-dimensional closed subscheme $Z\subset X$ of length $k+1.$ 
	\end{defi}
	
	\begin{rmk}\label{rmk:k-very}
		Observe that a vector bundle is $0$-very ample if and only if it is globally generated. For line bundles, $1$-very ampleness is equivalent to very ampleness. In general, if $k>0$ and $\F$ is a $k$-very ample vector bundle on a (smooth) projective variety, then $\F$ is ample (and $\det(\F)$ is very ample) \cite[Remark 1.3 \& Lemma 1.4]{ballico1992higher}.
	\end{rmk}
	
	\begin{defi}\label{def:augmented}
		Let $\F$ be a vector bundle on a projective variety $X.$ The \emph{base locus} of $\F$ is the set $$\Bs(\F)=\left\{x\in X\ |\ H^0(X,\F)\to\F(x)\ \text{is not surjective}\right\}.$$ The \emph{stable base locus} of $\F$ is defined as \[
		\B(\F)=\bigcap_{k\geq1}\Bs(\Sym^k\F).
		\] The \emph{augmented} (resp. \emph{restricted}) \emph{base locus} of $\F$ is $$\B_+(\F)=\bigcap_{k\geq1}\B((\Sym^k\F)(-A))\hspace{1cm} \ \left(\mathrm{resp.}\ \B_-(\F)=\bigcup_{k\geq1}\B((\Sym^k\F)(A))\right),$$ where $A$ is an ample line bundle on $X.$
	\end{defi}
	
	\begin{rmk}\label{rmk:augmented}
		One can prove that the definition of augmented (and restricted) base locus of a vector bundle $\F$ does not depend on the choice of the ample line bundle. Moreover, if $\pi\colon\P(\F)\to X$ denotes the natural projection, we have $\pi(\B_+(\O_{\P(\F)}(1)))=\B_+(\F)$ \cite[Proposition 6.4]{fulger2021seshadri}.
	\end{rmk}
	
	\begin{defi}\label{def:V-big}
		A vector bundle $\F$ on a projective variety $X$ is said \emph{V-big} if there exist an ample line bundle $A$ and a positive integer $k>0$ such that $\b{\B_-((\Sym^k\F)(-A))}\neq X.$
	\end{defi}
	
	\begin{thm}[\textbf{\cite[Theorem 6.4]{bauer2015positivity}}]\label{thm:V-big}
		Let $X$ be a smooth projective variety and let $\F$ be a vector bundle on $X.$ Then $\F$ is V-big if and only if $\B_+(\F)\neq X.$
	\end{thm}
	
	\begin{rmk}\label{rmk:V-big} 
		On smooth projective varieties, a V-big vector bundle is big \cite[Corollary 6.5]{bauer2015positivity}. The converse is true for line bundles \cite[Proposition-Definition 4.2]{bauer2015positivity}, but it is typically false for higher ranks (see for instance \cite[Remark 6.6]{bauer2015positivity}).
	\end{rmk}
	
	We now come to some generalities on Ulrich vector bundles.
	
	Just like in \cite[Theorem 2.3]{beauville2018introduction}, one can prove the following characterization of $B$-Ulrich bundles.
	
	\begin{thm}\label{thm:def-ulrich}
		Let $X$ be a smooth projective variety of dimension $n\geq1$ and let $B$ be a globally generated ample line bundle. Let $\E$ be a vector bundle of rank $r.$
		The following conditions are equivalent:
		\begin{itemize}
			\item[(1)] $\E$ is $B$-Ulrich.
			\item[(2)] There exists a linear resolution 
			\[
			\xymatrix{0 \ar[r] & L_c \ar[r] & L_{c-1}\ar[r] &\cdots\ar[r] & L_0\ar[r] & \phi_* \E\ar[r] &0},
			\]
			where $\phi=\phi_B\colon X\rightarrow\P^N,$ and $L_i=\O_{\P^{N}}(-i)^{\oplus b_i}$ and $c=N-n.$
			\item[(3)] For all finite surjective morphism $\pi\colon X\rightarrow \P^n$ such that $\pi^\ast\O_{\P^n}(1)\cong B,$ the sheaf $\pi_\ast \E$ is isomorphic to $\O_{\P^n}^{\oplus rd},$ where $d=B^n.$
		\end{itemize}
	\end{thm}
	
	The following properties are well known for Ulrich bundles with respect to a very ample polarization. They continue to hold in this slightly wider setting and are proven in the same way (possibly with some minor modification). 
	
	\begin{lemma}\label{lem:claim}
		Let $X$ be a smooth projective variety of dimension $n\geq 1$ and let $B$ a globally generated ample line bundle with $d=B^n.$ Suppose $X$ carries a $B$-Ulrich vector bundle $\E$ of rank $r.$ Then:
		\begin{itemize}\setlength\itemsep{0em}
			\item [(i)] $\E$ is $0$-regular (in the sense of Castelnuovo-Mumford), globally generated and arithmetically Cohen-Macaulay.
			\item [(ii)] $\chi(X,\E(mB))=\frac{rd}{n!}(m+1)\cdots(m+n)$ and $h^0(X,\E)=rd.$
			\item [(iii)] $c_1(\E)\cdot B^{n-1}=\frac{r}{2}\left((n+1)B^n+K_X\cdot B^{n-1} \right).$
			\item [(iv)] If $n\geq 2,$ then $$c_2(\E)\cdot B^{n-2}=\frac{1}{2}\left[c_1(\E)^2-c_1(\E)\cdot K_X\right]\cdot B^{n-2}+\frac{r}{12}\left[K_X^2+c_2(X)-\frac{3n^2+5n+2}{2}B^2\right]\cdot B^{n-2}.$$
			\item [(v)] If $Y\subset X$ is a smooth irreducible member in $|B|,$ then $\E_{|Y}$ is a $B_{|Y}$-Ulrich bundle on $Y.$ 
			\item[(vi)] There exists a $B^{\o k}$-Ulrich vector bundle of rank $r\cdot n!$ for all $k\geq 1.$
			\item[(vii)] If $f\colon X\to Z$ is a finite surjective morphism of smooth projective varieties such that $B\cong f^\ast L$ for a globally generated ample line bundle $L$ on $Z,$ then $f_\ast\E$ is a $L$-Ulrich bundle on $Z.$
			\item [(viii)] $\E$ is semistable and if it is not stable, then $\E$ is extension of $B$-Ulrich vector bundles of smaller rank. Moreover	$\E$ is stable if and only if it is $\mu$-stable. 
			\item[(ix)] $\O_X(kB)$ is $B$-Ulrich if and only if $(X,B,k)=(\P^n,\OPn(1),0).$
		\end{itemize}
	\end{lemma}
	\begin{proof}
		For (i)-(v) see (the proof of) \cite[Lemma 3.2]{lopez2023varieties} and references therein. For (vi)-(viii) see \cite[(3.2)-(3.6) \& Corollary 3.2]{beauville2018introduction} and \cite[Theorem 2.9]{casanellas2012stable}. For (ix) see \cite[Lemma 4.2(vi)]{antonelli2023varieties}.
	\end{proof}
	
	We conclude this section with a couple of examples of Ulrich vector bundles that will be useful later.
	
	First we recall the following definitions.
	
	\begin{defi}
		Let $X$ be a smooth projective variety and let $B$ a globally generated ample line bundle on $X.$ A vector bundle $\E$ of rank $2$ is said to be \emph{special with respect to $B$} if $\det(\E)\cong K_X((n+1)B).$
	\end{defi}
	
	\begin{defi}\label{def:DP}
		A \emph{Del Pezzo manifold of degree $d$} is a polarized smooth projective variety $(X,B)$ of dimension $n\geq 2$ such that $K_X=-(n-1)B$ and $B^n=d.$ When $n=2,$ we set $(X,B)=(S,-K_S).$ These varieties are classified in \cite[Theorem 3.3.1]{iskovskikh1999fano}: if $d\geq 3,$ the polarization $B$ determines an embedding in $\P^{d+n-2};$ for $d=2,$ then $\phi_B\colon X\rightarrow\P^n$ is a double cover branched over a hypersurface of degree $4;$ finally, $B$ has a single base point if $d=1$ \cite[Proposition 3.2.4]{iskovskikh1999fano}.
	\end{defi}
	
	In \cite[Propositions 4.1(i)-6.1]{beauville2018introduction} the author proves that all Del Pezzo surfaces and threefolds $(X,B)$ of degree $d\geq3$ carry an Ulrich vector bundle for the hyperplane section $\O_X(B).$ We extend these results to those of degree $d=2$ when the polarization $B$ is simply base-point-free.
	
	\begin{prop}\label{prop:ulrich-DP-surface}		
		Let $S$ be a Del Pezzo surface of degree $2,$ and let $L$ be a line bundle such that $L^2=-2$ and $L\cdot K_S=0.$ Then $\E=L(-K_S)$ is a $(-K_S)$-Ulrich line bundle on $S.$ Since $S$ is isomorphic to the blow-up of $\P^2$ at seven points in general positions, we can take $L$ as the difference of two distinct exceptional divisors.
	\end{prop}
	\begin{proof}
		This is a check on the vanishing which does not involve the very ampleness of $-K_S.$ Hence we refer to \cite[Proposition 4.1(i)]{beauville2018introduction}.
	\end{proof}
	
	\begin{prop}\label{prop:ulrich-DP-threefold}
		A Del Pezzo threefold $(X,B)$ of degree $2$ carries a stable special $B$-Ulrich vector bundle of rank $2.$
	\end{prop}
	\begin{proof}
		In \cite[Proof of Theorem D, Step 1]{faenzi2014even} it is shown that $X$ contains a smooth elliptic curve $\G$ of $B$-degree $4$ such that $H^0(X,\I_{\G/X}(B))=0.$ Using Hartshorne-Serre correspondence and proceeding as in \cite[Proposition 6.1]{beauville2018introduction}, one finds a rank $2$ special vector bundle $\E$ which is also $B$-Ulrich. Since $\Pic(X)=\Z\cdot\O_X(B)$ \cite[Remark 3.3.2(i)]{iskovskikh1999fano}, $X$ cannot support $B$-Ulrich line bundles by Lemma \ref{lem:claim}(ix). It follows that $\E$ must be stable by Lemma \ref{lem:claim}(viii).
	\end{proof}
	
	\section{General facts about Seshadri constants}\label{sect:seshadri}
	This section is devoted to recall the main properties of Seshadri constants. For a detailed account we refer to \cite{bauer2009primer}. Finally we will also prove a generalization of \cite[Lemma 7.1]{lopez2022positivity} which shows the existence of a Seshadri curve in a special case.
	
	\begin{defi}
		Let $X$ be a projective variety. Let $Y\subset X$ be a closed subscheme and let $L$ be a nef line bundle on $X.$ The \emph{Seshadri constant of $L$ at $Y$} is the non-negative real number 
		\[
		\eps(X,L;Y)=\eps(L;Y)=\sup\left\{\eps\geq0\ |\ \mu^\ast L-\eps E\ \mbox{is nef}\right\},
		\] where $\mu\colon\wX\to X$ is the blow-up at $X$ along $Y$ with exceptional divisor $E.$
		
		If $Y$ is supported at $q$ distinct points $x_1,\dots,x_q,$ then
		\[
		\eps(L;Y)=\eps(L;x_1,\dots,x_q)
		\] is called the \emph{multi-point Seshadri constant of $L$ at $x_1,\dots,x_q.$} The real number 
		\[
		\eps(L)=\eps(X,L):=\inf_{x\in X}\eps(X,L;x)
		\] is the \emph{Seshadri constant of $L.$}
	\end{defi}
	
	For a nef line bundle on a smooth projective variety, Nakamaye theorem leads to a characterization in terms of Seshadri constants of its augmented base locus.
	
	\begin{rmk}[\textbf{\cite[Remark 6.5]{ein2009restricted}}]\label{rmk:augmented-seshadri}
		Let $X$ be a smooth projective variety of dimension $n\geq1$ and let $L$ be a nef line bundle on $X.$ Then $$\B_+(L)=\left\{x\in X\ |\ \eps(L;x)=0\right\}.$$
	\end{rmk}
	
	\begin{prop}\label{prop:seshadri-properties}
		Let $X$ be a projective variety and let $L$ be a nef line bundle. Fix a point $x\in X.$ Then:
		\begin{itemize}\setlength\itemsep{0em}
			\item[(1)] $\eps(L;x)$ depends only on the numerical equivalence class of $L$ and satisfies $\eps(mL;x)=m\cdot\eps(L;x)$ for every $m\in\mathbf{N}.$
			\item[(2)] If $V\subset X$ is a subvariety of dimension $k\geq1$ passing through $x,$ then 
			\[
			\eps(L;x)\leq \left(\frac{(L^k\cdot V)}{\mult_x(V)}\right)^{\frac{1}{k}}.
			\] Moreover, equality holds for some $V,$ possibly equal to $X,$ passing through $x.$
			\item[(3)] If $L$ is ample and base-point-free, then $\eps(L;x)\geq1.$
			\item[(4)] The \emph{Seshadri function} $\eps(L;-)\colon X\to\R_+$ attains a maximal value for a very general point. We denote this value by $\eps(X,L;1)=\eps(L;1).$ 
		\end{itemize}
	\end{prop}
	\begin{proof}
		For (1)-(2)-(3) see \cite[Examples 5.1.3-5.1.18 \& Proposition 5.1.9]{lazarsfeld2017positivity}. Finally (4) immediately follows from \cite[Example 5.1.11]{lazarsfeld2017positivity} applied to the projection onto the first factor $p\colon X\times X\to X,$ with section given by the diagonal morphism $\Delta\colon X\to X\times X,$ and line bundle $\L=p^\ast L.$
	\end{proof}
	
	Recently the definition of Seshadri constant at a point has been generalized to vector bundles in \cite{fulger2021seshadri}. Actually the setting can be even more general. However, for our purposes, it is enough to stick to the case of locally free sheaves.
	
	\begin{defi}\label{def:seshadri-vector}
		Let $X$ be a projective variety and let $\F$ be a vector bundle on $X.$ Let $\pi\colon\P(\F)\to X$ be the natural projection and let $\xi=\O_{\P(\F)}(1)$ be the tautological line bundle on $\P(\F).$ Let $x\in X$ be a point and let $\mathcal{C}_{\F,x}$ denote the set of irreducible curves in $\P(\F)$ that meet $\P(\F(x))$ but are not contained in the fibre $\P(\F(x)).$ The \emph{Seshadri constant of $\F$ at $x$} is 
		\[
		\eps(\F;x):=\inf_{C\in\mathcal{C}_{\F,x}}\left\{\frac{\xi\cdot C}{\mult_{x}(\pi_\ast C)}\right\}.
		\]
	\end{defi}
	
	Whenever the vector bundle is nef, the definition can be formulated analogously to the one of line bundles.
	
	\begin{rmk}[\textbf{\cite[Remark 3.10(a-c)]{fulger2021seshadri}}]\label{rmk:seshadri-vector}
		Let $X$ be a projective variety and let $\F$ be a nef vector bundle on $X.$ Let $x\in $ be a point and let $\mu\colon \wX\to X$ be the blow-up at $x$ with exceptional divisor $E.$ Then:
		\begin{itemize}
			\item[(a)] $\eps(\F;x)\geq0.$
			\item[(b)] $\eps(\F;x)=\sup\left\{\eps\geq0\ |\ \mu^\ast\F\langle-\eps E\rangle\ \text{is nef on}\ \wX\right\}.$
		\end{itemize}
	\end{rmk}
	
	\begin{rmk}\label{rmk:augmented-vector}
		Just like for nef line bundles, for a nef vector bundle $\F$ on a projective variety $X,$ one has the following characterization \cite[Proposition 6.9]{fulger2021seshadri}:\[
		\B_+(\F)=\{x\in X\ |\ \eps(\F;x)=0\}.
		\] Combining this fact with \cite[Corollary 6.10]{fulger2021seshadri}, we immediately deduce that a nef vector bundle $\F$ is ample if and only if $\B_+(\F)=\emptyset.$
	\end{rmk}
	
	These objects share many features with the classical Seshadri constants for line bundles, such as the Seshadri ampleness criterion \cite[Theorem 3.11]{fulger2021seshadri}. For our goals, we will only use these basic properties. Therefore, in the remaining part of the section, we will focus on Seshadri constants for line bundles.	
	
	The following results are well known and standard, hence we will only sketch their proof.
	
	\begin{prop}\label{prop:seshadri-multipoint}
		Let $X$ be a projective variety of dimension $n\geq 1$ and let $L$ be a nef line bundle. 
		\begin{itemize}\setlength\itemsep{0em}
			\item[(1)] If $x_1,\dots,x_q$ are $q$ distinct points, then
			\[
			\eps(L;x_1,\dots,x_q)=\inf\left\{ \frac{(L\cdot C)}{\sum_{i=1}^{q}\mult_{x_{i}}(C)}\ |\ C\subset X\ \mbox{is an irreducible curve passing through some}\ x_i \right\}.
			\]
			
			\item[(2)] If $x_1,\dots,x_q$ are smooth points for $X,$ then 
			\begin{equation}\label{eq:seshadri-lower}
				\eps(L;x_1,\dots,x_q)\geq\frac{1}{q}\min_{j=1,\dots,q}\eps(L;x_i)\geq\frac{1}{q}\eps(L)
			\end{equation} 
			and
			\begin{equation}\label{eq:seshadri-submaximal}
				\eps(L;x_1,\dots,x_q)\leq \left(\frac{L^{n}}{q}\right)^{\frac{1}{n}}.
			\end{equation}
			\item[(3)] If $x\in X$ is a smooth point and $Z\subset X$ is a $0$-dimensional closed subscheme which is smooth at $x\in Z,$ then $\eps(L;Z)\leq\eps(L;x).$
		\end{itemize}
	\end{prop}
	\begin{proof}
		The blow-up $\mu\colon\w{X}\rightarrow X$ of $X$ along $\{x_1,\dots,x_q\}$ can be seen as the composition $$\mu=\mu_q\circ\cdots\circ\mu_1\colon\w{X}=X_q\rightarrow X_{q-1}\rightarrow\cdots\rightarrow X_1\rightarrow X_0=X,$$ where $\mu_j,$ for $1\leq j\leq q,$ is the blow-up of $X_{j-1}$ centred at the strict transform $x^{(j)}\in X_{j-1}$ of $x_j$ via $\mu_{j-1}.$ Therefore (1) follows by applying inductively \cite[Proposition 5.1.5]{lazarsfeld2017positivity}.		Part (2) is \cite[(6.4.7) \& Proposition 2.1.1]{bauer2009primer}.	For (3), let $\rho\colon X'\to X$ denote the blow-up of $X$ at $x$ with exceptional divisor $F,$ and let $Z'\subset X'$ be the strict transform of $Z_1=Z-\{x\}.$ Let $\rho'\colon \bX\to X'$ be the blow-up of $X'$ along $Z'$ with exceptional divisor $F_{Z'}.$ The composition $\lambda=\rho'\circ \rho$ is the blow-up of $X$ along $Z$ with exceptional divisor $\b F=F'+F_{Z'},$ where $F'$ is the strict transform of $F$ via $\rho'.$ Let $\eps\geq0$ such that $\lambda^\ast L-\eps \b F$ is nef. To obtain the conclusion, it is enough to prove that $\eps\leq\frac{L\cdot C}{\mult_x(C)}$ for every irreducible curve $C\subset X$ passing through $x.$ Indeed, if this is true, by taking the infimum over all irreducible curve passing through $x$ we get $\eps\leq\eps(L;x)$ by part (1). Then, by taking the sup over all $\eps\geq0$ such that $\lambda^\ast L-\eps \b F$ is nef, we get $\eps(L;Z)\leq\eps(L;x).$\\
		Given an irreducible curve $C\subset X$ containing $x,$ set $C'\subset X'$ and $\b C\subset\bX$ to be respectively the strict transform of $C$ via $\rho$ and the strict transform of $C'$ via $\rho'.$ In particular $\b C$ is the strict transform of $C$ through $\pi,$ and  $C'\nsubseteq F$ and $\b C\nsubseteq\b F$ otherwise they would be mapped onto a point. As a consequence $F_{Z'}\cdot\b C\geq 0.$ Therefore, using also projection formula and \cite[Lemma 5.1.10]{lazarsfeld2017positivity}, we obtain:
		\[
		0\leq(\lambda^\ast L-\eps\b F)\cdot \b C=\rho'^\ast\left(\rho^\ast L-\eps F\right)\cdot\b C-\eps F_{Z'}\cdot\b C\leq(\rho^\ast L-\eps F)\cdot C'=L\cdot C-\eps\cdot \mult_x(C),
		\] which gives the claim.
	\end{proof}
	
	\begin{defi}
		Let $X$ be a projective variety and let $L$ be a nef line bundle. Let $x_1,\dots,x_q\in X$ be  $q$ distinct points. An irreducible curve $C\subset X$ realizing the infimum in Proposition \ref{prop:seshadri-multipoint}(1) is called \emph{Seshadri curve of $L$ at $x_1,\dots,x_q.$} Assuming that $x_1,\dots,x_q$ are smooth, an irreducible curve $D\subset X$ such that 
		\[
		\frac{(L\cdot D)}{\sum_{i=1}^q\mult_{x_i}(D)}\leq \left(\frac{L^{\dim X}}{q}\right)^{\frac{1}{\dim X}}
		\] is said to be \emph{submaximal for $L$ at $x_1,\dots,x_q.$} If the inequality in (\ref{eq:seshadri-submaximal}) is strict, we say that $L$ is \emph{submaximal at $x_1,\dots,x_q.$}
	\end{defi}
	
	\begin{rmk}\label{rmk:exceptional}
		Let $\mu\colon\wX\to X$ be the blow-up of a projective variety $X$ along a $0$-dimensional closed subscheme $Y\subset X$ and let $E$ be the exceptional divisor. Then the line bundle $\O_E(-E)$ is ample.
		
		Indeed, $\O_{\wX}(-E)$ is $\mu$-ample (see, e.g., \cite[Proposition 13.96(1)]{gortz2010algebraic}), hence $\O_E(-E)$ is $\mu_{|E}$-ample by \cite[Example 1.7.7]{lazarsfeld2017positivity}. Since every non-empty subvariety $V\subset E$ is contracted to a point,  Nakai-Moishezon-Kleiman criterion for mapping \cite[Corollary 1.7.9]{lazarsfeld2017positivity} implies that $(-E)^{\dim V}\cdot V>0,$ which says that $\O_E(-E)$ is ample by Nakai-Moishezon-Kleiman theorem \cite[Theorem 1.2.23]{lazarsfeld2017positivity}.
	\end{rmk}
	
	\begin{lemma}\label{lem:seshadri-subscheme2}
		Let $X$ be a projective variety of dimension $n\geq1$ and let $Y\subset X$ be a non-empty proper closed subscheme. Let $\mu\colon\wX\rightarrow X$ be the blow-up of $X$ along $Y$ with exceptional divisor $E,$ and let $L$ be an ample line bundle on $X.$ Then $\eps(L;Y)=\sup\left\{ \l\geq0\ |\ \mu^\ast L-\l E\ \mbox{is ample} \right\}.$
	\end{lemma}
	\begin{proof}
		Write $\eps=\eps(L;Y),$ $\w{L}_\eps=\mu^\ast L-\eps E.$ Let $\eps'$ be the supremum of the set $\left\{ \l>0\ |\ \mu^\ast L-\l E\ \mbox{is ample} \right\},$ which is non-empty by \cite[Exercise II.7.14(b)]{hartshorne2013algebraic}. An ample line bundle is nef, hence $\eps'\leq \eps$ is clear. Once recalled that the ample cone $\w A=\mathrm{Amp}(\w{X})$ of $\wX$ is the interior of the nef cone $\mathrm{Nef}(\w{X}),$ which is in turn the closure of $\w A,$ the other direction immediately follows by \cite[Theorem 6.1]{rockafellar1997convex}.
	\end{proof}
	
	The next remark shows that the separation properties of a line bundle influence the Seshadri constant.
	
	\begin{rmk}\label{rmk:seshadri-jet}
		Let $X$ be a projective variety and let $L$ be a nef line bundle on $X.$ Let $x\in X$ be a point and suppose that $L$ separates tangent vectors at $x,$ namely the restriction map $H^0(X,L)\to H^0(Z,L_{|Z})$ to every $0$-dimensional closed subscheme $Z\subset X$ of length $2$ with $\supp(Z)=\left\{x\right\}$  is surjective (see, for instance, the proofs in \cite[\href{https://stacks.math.columbia.edu/tag/0E8R}{Tag 0E8R}]{stacks-project}). Then $\eps(L;x)\geq1.$
		
		To prove this we use the characterization in Proposition \ref{prop:seshadri-multipoint}(1). Let $C\subset X$ be an irreducible curve and let $v\in T_xC\subset T_xX$ be a non-zero tangent vector. Since $L$ separates tangents at $x,$ we can find a divisor $D\in |L|$ which passes through $x$ with $v\notin T_xD$ \cite[Remark II.7.8.2]{hartshorne2013algebraic}. It follows that $C\not\subset D,$ for otherwise $v\in T_xC\subset T_xD$ giving a contradiction. Then
		\[
		(L\cdot C)=(D\cdot C)\geq \mult_x(C).
		\] This implies  that $\eps(L;x)\geq1,$ as claimed.
	\end{rmk}
	
	In the smooth case, generally the existence of a Seshadri curve is known on surfaces for submaximal line bundles (see \cite[Proposition 1.1]{bauer2008seshadri} and \cite[§1, p.2, lines 9-15]{hanumanthu2017multi}). A special case is \cite[Lemma 7.1]{lopez2022positivity} which says that if $H$ is a very ample line bundle on a smooth projective variety, then $\eps(H;x)=1$ if and only if there is a line passing through $x.$ This means that the line is the Seshadri curve for $H$ at $x.$ 
	
	The goal of this section is to extend \cite[Lemma 7.1]{lopez2022positivity} to the case of a base-point-free ample line bundle. 
	
	\begin{lemma}\label{lem:7.1}
		Let $X$ be a smooth projective variety of dimension $n\geq1,$ let $B$ be a globally generated ample line bundle on $X.$ Let $x\in X$ be a point and let $\mu\colon\w{X}\rightarrow X$ be the blow-up at $x$ with exceptional divisor $E.$ Then the following conditions are equivalent:
		\begin{itemize}\setlength\itemsep{0em}
			\item[(i)] $\eps(B;x)=1.$
			\item[(ii)] $\mu^\ast B-E$ is not ample on $\w{X}.$
			\item[(iii)] $\eps(B;x)=1$ and there exists a Seshadri curve $\G\subset X$ for $B$ at $x$ mapping finitely onto a line $L\subset\P^N$ through $(\phi_B)_{|\G}\colon\G\rightarrow L$ with $x$ as point of maximum multiplicity for $\G$ and as unique point in the fibre over $(\phi_B)_{|\G}(x).$
		\end{itemize}
	\end{lemma}
	
	\begin{proof}
		Clearly (iii) implies (i). If $\eps(B;x)=1,$ then $\mu^\ast B-E$ cannot be ample due to the openness of the ample cone: if $\mu^\ast B-E$ was ample, by \cite[Example 1.3.14]{lazarsfeld2017positivity} the $\R$-divisor $\mu^\ast B-(1+\sigma)E$ would be ample for some $\sigma>0,$ yielding $\eps(B;x)>1.$ Hence (i) implies (ii). 
		To conclude the proof, suppose that $\wB=\mu^\ast B-E$ is not ample. Since $H^0(\wX,\wB)=H^0(X,B\o\m_x)$ \cite[Lemma 4.3.16]{lazarsfeld2017positivity}, and the base scheme of the linear series of $|B\o\m_x|$ is the schematic fibre $\phi_B^{-1}(\phi_B^{\hspace{0.00000000001mm}}(x)),$ we deduce that the base scheme $Y=\Bs(|\wB|)$ is the union of the finite set of points $\left\{\mu^{-1}(y)\ |\ y\in \phi_B^{-1}(\phi_B^{\hspace{0.00000000001mm}}(x)),y\neq x \right\}$  with a subvarietiety $V\subset E.$ As $\O_E(-E)\simeq\O_{\P^1}(1),$ it easily follows that $\wB_{|Y}$ is ample. Therefore $\wB$ is semiample by Zariski-Fujita theorem \cite[Theorem 1.10]{fujita1983semipositive}. On the other hand, $\wB$ cannot be strictly nef, for otherwise it would be ample. Therefore there exists an irreducible curve $\G'\subset \wX$ such that $\wB\cdot\G'=0.$ Since $-E_{|E}$ is ample, $\G'$ is not contained in $E.$ Moreover, the irreducible curve $\G=\mu(\G')$ must contain $x,$ otherwise we would have $(\mu^\ast B-E)\cdot\G'=B\cdot\G>0.$ It follows from \cite[Corollary II.7.15]{hartshorne2013algebraic} that $\G'$ is blow-up of $\G$ at $x.$ Therefore, by \cite[Lemma 5.1.10]{lazarsfeld2017positivity}, we get 
		\[
		0=(\mu^\ast B-E)\cdot\G'=B\cdot\G-\mult_x(\G),
		\] which says, together with Proposition \ref{prop:seshadri-properties}(3), that $\eps(B;x)=1$ and that $\G$ is a Seshadri curve for $B$ at $x.$ Now, consider the restriction $\phi'=(\phi_B)_{|\G}\colon\G\to\phi_B(\G)=:L\subset\P^N$ and let $H\subset \P^N$ be a hyperplane. A consequence of Zariski's formula for finite extensions \cite[Equation (2.2)]{abad2020finite}, combined with Proposition \ref{prop:seshadri-properties}(3) for $H$ and projection formula, implies $$\mult_x(\G)\le \max_{y\in \G}\mult_y(\G)\le \deg (\phi')\cdot \max_{z\in L}\mult_{z}(L)\leq \deg(\phi')\cdot(H\cdot L)=B\cdot\G=\mult_x(\G).$$ We deduce that all of these are equalities. In particular, one has $H\cdot L=\mult_{z}(L)$ for some $z\in L,$ forcing $L$ to be a line (e.g. by \cite[Lemma 7.1]{lopez2022positivity}), and
		\begin{equation}\label{eq:lemma7.1}
			B\cdot\G=\mult_x(\G)=\max_{y\in \G}\mult_y(\G)=\deg(\phi')=\deg(\phi')\cdot\mult_{\phi'(x)}(L).
		\end{equation} This forces $x$ to be the only point in the fibre over $\phi'(x)$: the inverse image of an affine open neighborhood of $\phi'(x)$ is an affine open neighborhood of $x$ due to the finiteness of $\phi',$ hence the claim follows by \cite[$(\ast)$ condition]{abad2020finite}. This proves (iii) and completes the proof.
	\end{proof}
	
	\begin{lemma}\label{lem:7.1-2}
		Let $X$ be a smooth projective variety of dimension $n\geq1$ and let $B$ be a globally generated ample line bundle on $X.$ Let $x\in X$ be a point and let $\rho\colon X'\to X$ be the blow-up of $X$ along the schematic fibre $\phi_B^{-1}(\phi_B^{\hspace{0.000001mm}}(x))$ with exceptional divisor $E'.$ Then $\eps(B;\phi_B^{-1}(\phi_B^{\hspace{0.000001mm}}(x)))\geq1.$ Moreover, if $\phi_B^{-1}(\phi_B^{\hspace{0.000001mm}}(x))$ is smooth, the following are equivalent:
		\begin{itemize}\setlength\itemsep{0em}
			\item[(i)] $\eps(B;\phi_B^{-1}(\phi_B^{\hspace{0.00000001mm}}(x)))=1.$
			\item[(ii)] $\rho^\ast B-E'$ is not ample.
			\item[(iii)] There exists an irreducible curve $C\subset X$ such that $B\cdot C=\sum_{j=1}^q\mult_{x_j}(C),$ where $\phi_B^{-1}(\phi_B^{\hspace{0.000001mm}}(x))=\{x_1,\dots,x_q\}.$
		\end{itemize}
	\end{lemma}
	\begin{proof}
		The line bundle $\rho^\ast B-E'$ is generated by global sections because $\rho$ is the blow-up along the base scheme of $|B\o\m_x|$ (see the proof of Lemma \ref{lem:7.1}). In particular it is nef. As the Seshadri constant is the supremum over all $\eps\geq0$ such that $\rho^\ast B-\eps E'$ is nef, we deduce that $\eps(B;\phi_B^{-1}(\phi_B^{\hspace{0.000001mm}}(x)))\geq1.$
		
		We now turn to prove the second part of the statement. Henceforth we suppose that $\phi_B$ is unramified at every point of $\phi_B^{-1}(\phi_B^{\hspace{0.000001mm}}(x)),$ which amounts to say that the schematic fibre over $\phi_B(x)$ is smooth. If (iii) holds, we get (i) by combining the first part and Proposition \ref{prop:seshadri-multipoint}(1). Moreover, item (ii) directly follows from (i) because of the openness of the ample cone (see the proof of Lemma \ref{lem:7.1}). 
		Now assume (ii). Since $\rho^\ast B-E'$ is base-point-free, to not get a contradiction, it cannot be strictly nef, for otherwise it would be ample by \cite[Corollary 1.2.15]{lazarsfeld2017positivity}. Therefore there is an irreducible curve $C'\subset X'$ such that $(\rho^\ast B-E')\cdot C'=0.$ The divisor $-E'_{|E'}$ is ample 
		(Remark \ref{rmk:exceptional}) and every subvariety in $E'$ is contracted to a point, hence $C'\nsubseteq E'.$ Moreover $C=\rho(C')$ passes through some $x_i,$ for otherwise we would have $(\rho^\ast B-E')\cdot C'=B\cdot C>0$ by the ampleness of $B.$ Writing $E'=E_1+\cdots+E_q,$ by projection formula and \cite[Lemma 5.1.10]{lazarsfeld2017positivity} we get
		\[
		0=(\rho^\ast B-E')\cdot C'=\rho^\ast B\cdot C'-\sum_{j=1}^qE_j\cdot C'=B\cdot  C-\sum_{j=1}^q\mult_{x_j}(C).
		\] Thus (iii) holds.
	\end{proof}
	
	We point out that the condition $\eps(B,\phi_B^{-1}(\phi_B^{\hspace{0.000000000000000001mm}}(x)))>1$ is open even without assuming the smoothness of the fibre.
	
	\begin{prop}\label{prop:1-seshadri}
		Let $B$ be an ample and globally generated line bundle on a smooth projective variety $X.$ 
		If $\eps(B;\phi_B^{-1}(\phi_{B}^{\hspace{0.000001mm}}(y)))>1$ for a point $y\in X,$ then the locus \[\left\{x\in X\ |\ \eps(B;\phi_B^{-1}(\phi_{B}^{\hspace{0.000001mm}}(x)))>1\right\}
		\] contains a dense open subset.
	\end{prop}
	\begin{proof}
		The proof follows \cite[Lemma 1.4]{ein1995local}. Write $\phi=\phi_B\colon X\to\P^N,$ and, for every point $x\in X,$ let $\mu_x\colon\wX_x\to X$ denote the blow-up along $\phi^{-1}(\phi(x))$ and let $E_x$ be the exceptional divisor. Denote by $\Delta$ the diagonal of $\P^N\times\P^N$ and let $b\colon Z\to X\times X$ be the blow-up of $X\times X$ along $(\phi\times\phi)^{-1}(\Delta)$ with exceptional divisor $E.$ Let $g=\pi_1\circ b$ and $h=\pi_2\circ b$ be respectively the composition of $b$ with the first and the second projection from $X\times X$ to $X.$ Denoting by $Z_x$ every schematic fibre $g^{-1}(x),$ we see that $Z_x\cong\wX_x$ and $g_{|Z_x}=\mu_x.$ In particular, the restrictions $(h^\ast B)_{|Z_x}$ and $E_{|Z_x}$ are respectively $\mu_x^\ast B$ and $E_x.$ Recall that $\mu_x^\ast B-E_x$ is generated by global section, as $\mu_x$ is the blow-up of the base scheme of the linear series $|B\o\m_x|.$ It follows that $\eps(B;\phi^{-1}(\phi(x)))\geq1.$
		
		If $\eps(B;\phi^{-1}(\phi(y)))>1,$ we claim that $\mu_y^\ast B-E_y$ is ample. By \cite[Corollary 1.2.15]{lazarsfeld2017positivity}, it is enough to prove that $\mu_y^\ast B-E_y$ is strictly nef. Let $C\subset\wX_y$ be an irreducible curve. If $E_y\cdot C<0,$ then $C\subset E_y,$ whence $C$ is contracted to a point. Consequently Nakai-Moishezon-Kleiman criterion for mapping \cite[Corollary 1.7.9]{lazarsfeld2017positivity} gives  \[
		(\mu_y^\ast B-E_y)\cdot C=-E_y\cdot C>0.
		\] If $E_y\cdot C\geq0,$ taking a real number $\eps>1$ such that $\mu_y^\ast B-\eps E_y$ is ample (Lemma \ref{lem:seshadri-subscheme2}),  we get \[
		(\mu_y^\ast B-E_y)\cdot C=(\mu_y^\ast B-\eps E_y)\cdot C + (\eps-1)E_y\cdot C>0.
		\] This proves the claim. Applying \cite[Theorem 1.2.17]{lazarsfeld2017positivity} to the pair $(g\colon Z\to X,h^\ast B-E),$ we can find an open dense subset $U\subset X$ of points such that $$(h^\ast B-E)_{|Z_x}=\mu_x^\ast B-E_x$$ is ample whenever $x\in U.$ The $\R$-divisor $\mu_x^\ast B-(1+\sigma)E_x$ remains ample for all sufficiently small $\sigma>0$ \cite[Example 1.3.14]{lazarsfeld2017positivity}, hence $\eps(B;\phi^{-1}(\phi(x)))>1$ for every $x\in U.$
	\end{proof}
	
	\begin{defi}
		Let $X$ be a smooth projective variety and let $B$ be a globally generated ample line bundle. 
		
		A \emph{$1$-Seshadri curve for $B$ at a point $x$} is an irreducible curve $\G\subset X$ such that $B\cdot\G=\mult_x(\G).$ We say that $X$ is \emph{covered by $1$-Seshadri curves for $B$} if for every point $y$ there is a $1$-Seshadri curve for $B$ at $y.$ 
		
		A  \emph{$1$-Seshadri curve for $\phi_B$ at $x$} is an irreducible curve $C\subset X$ such that $\phi_B^{-1}(\phi_B^{\hspace{0.000001mm}}(x))$ is smooth and
		$B\cdot C=\sum_{j=1}^q\mult_{x_j}(C),$ where $\phi_B^{-1}(\phi_B^{\hspace{0.000001mm}}(x))=\{x=x_1,\dots,x_q\}.$ We say that $X$ is \emph{generically covered by $1$-Seshadri curves for $\phi_B$} if for a general point $y\in X$ there exists a $1$-Seshadri curve for $\phi_B$ at $y.$
	\end{defi}
	
	\begin{rmk}\label{rmk:covered-1-seshadri}
		Let $X$ be a smooth projective variety and let $B$ be a globally generated ample line bundle on $X.$ If $X$ is not generically covered by $1$-Seshadri curves for $\phi_B,$ then there exists a point $x\in X$ such that $\phi_B^{-1}(\phi_B^{\hspace{0.000001mm}}(x))$ is smooth and $\eps(B;\phi_B^{-1}(\phi_B^{\hspace{0.000001mm}}(x)))>1.$ Actually, there is a dense open subset $V\subset X$ such that the schematic fibre over $\phi_B(y)$ is smooth and there is no $1$-Seshadri curve for $\phi_B$ at $y,$ for every $y\in V.$
		
		To see this, let $U=\phi_B^{-1}(\phi_B(X)-\Ram(\phi_B))$ be the (dense) open subset where $\phi_B^{-1}(\phi_B^{\hspace{0.000001mm}}(y))$ is smooth for every $y\in U,$ and suppose, by contradiction, that $\eps(B;\phi_B^{-1}(\phi_B^{\hspace{0.000001mm}}(y)))\leq1$ for every $y\in U.$ Lemma \ref{lem:7.1-2} implies that $\eps(B;\phi_B^{-1}(\phi_B^{\hspace{0.000001mm}}(y)))=1$ and, furthermore, that there exists a $1$-Seshadri for $\phi_B$ at $y,$ for every $y\in U.$ On the other hand, as $X$ is not generically  covered by $1$-Seshadri curves for $\phi_B,$ every non-empty open subset $W\subset X$ contains a point $z\in W$ such that either $\phi_B^{-1}(\phi_B^{\hspace{0.000001mm}}(z))$ is singular, or does not admit a $1$-Seshadri curve for $\phi_B$ at $z.$ As $\phi_B^{-1}(\phi_B^{\hspace{0.000001mm}}(y))$ is smooth for every $y\in U,$ by taking $W=U$ we get a contradiction. Therefore there must exist $x\in U$ such that $\eps(B;\phi_B^{-1}(\phi_B^{\hspace{0.000001mm}}(x)))>1$ as claimed. The conclusion follows by combining the first part with Proposition \ref{prop:1-seshadri} and Lemma \ref{lem:7.1-2}.
	\end{rmk}
	
	In the last part of the section we make some observations regarding Lemma \ref{lem:7.1}.	In a sort of analogy with \cite[Lemma 7.1]{lopez2022positivity}, we will see in Remark \ref{rmk:seshadri-B-line} that if $\phi_B$ is unramified at $x,$ then the $1$-Seshadri curve is a \emph{$B$-line} (Definition \ref{def:B-line}).
	
	The next easy well-known lemma will be helpful in many situations. 
	
	\begin{lemma}\label{lem:degree-1}
		Let $Y$ be a projective variety of dimension $m\geq 1$ and let $L$ be a globally generated ample line bundle such that $L^m=1.$ Then $(Y,L)\cong(\P^m,\O_{\P^m}(1))$ via $\phi_L.$
	\end{lemma}
	
	It's an immediate consequence of Lemma \ref{lem:degree-1} that every $B$-line is smooth and rational.
	
	\begin{rmk}\label{rmk:B-line}
		Let $X$ be a smooth projective variety and let $B$ an ample and globally generated line bundle. Let $C\subset X$ be a $B$-line and let $V\subset H^0(X,B)$ be any subspace such that $|V|$ is base-point-free. Then $C$ is an irreducible smooth rational curve which is mapped isomorphically onto $\P^1$ by $(\phi_V)_{|C}.$ In particular, a $B$-line for a very ample line bundle $B$ is an actual line in $X\subset\P^N$ via the embedding determined by $|B|.$
		
		Indeed, irreducibility is immediate from Nakai-Moishezon-Kleiman criterion \cite[Theorem 1.2.23]{lazarsfeld2017positivity}. Now, $B_{|C}$ is a globally generated ample line bundle on an irreducible curve having $\deg(B_{|C})=1.$ Therefore $\phi_{B_{|C}}$ induces an isomorphism with $\P^1$ by Lemma \ref{lem:degree-1}. Consider the subspace $V_C=\Imm\left(V\to H^0(C,B_{|C})\right)$ determined by the restriction of sections. We prove that $(\phi_V)_{|C}=\phi_{B_{|C}}.$ It's clear that $|V_C|$ is base-point-free, for otherwise there would be a point $x\in C$ such that $0=s_{|C}(x)=s(x)$ for every $s\in V$ yielding a contradiction. The morphism $\phi_V\colon X\to\P^N$ is finite, for otherwise $B_{|Z}\cong(\phi_V^\ast\OPN(1))_{|Z}$ would be trivial, hence non-ample, for every subvariety $Z\subset X$ of positive dimension which is contracted to a point. In particular $\dim V_C\geq 2.$ As $V_C\subset H^0(C,B_{|C})\cong H^0(\P^1,\O_{\P^1}(1))\cong\C^2,$ we must have $V_C=H^0(C,B_{|C}).$
	\end{rmk} 
	
	\begin{rmk}\label{rmk:seshadri-B-line}
	Let $X$ be a smooth projective variety and let $B$ a base-point-free ample divisor. Suppose there is a point $x\in X$ such that $\eps(B;x)=1$ and let $\G\subset X$ be a $1$-Seshadri curve for $B$ at $x$ (Lemma \ref{lem:7.1}). If $\phi_B$ is unramified at $x,$ then $\G$ is a $B$-line. However the converse does not hold.
		
		To prove this, first set $\phi=\phi_B$ and $\phi'=(\phi_{B})_{|\G}.$ If $\phi$ is unramified at $x,$ then so is $\phi',$ as the inclusion $\iota\colon\G\hookrightarrow X$ is unramified \cite[\href{https://stacks.math.columbia.edu/tag/02GC}{Tag 02GC}]{stacks-project}. Therefore,  using (\ref{eq:lemma7.1}), \cite[Theorem 2.28]{shafarevich1994basic} and Lemma \ref{lem:7.1}.3, we immediately get 
		\[
		B\cdot\G=\deg(\phi')=\#\left(\left(\phi'\right)^{-1}(\phi(x))\right)=\#\{x\}=1.
		\] A counterexample for the opposite implication is given by Del Pezzo threefolds of degree $2$: for every point there is a $B$-line passing through \cite[Proposition III.1.4(ii)]{iskovskikh1980anticanonical}, but $\phi$ has non-empty ramification locus (see Definition \ref{def:DP}).
	\end{rmk}

	The morphism induced by a very ample line bundle is unramified at all points. Therefore, Remark \ref{rmk:seshadri-B-line} shows that Lemma \ref{lem:7.1} reduces to \cite[Lemma 7.1]{lopez2022positivity} when $B$ is very ample.
	
	However, not all $1$-Seshadri curves are $B$-lines.
	
	\begin{rmk}\label{rmk:seshadrik3}
		A result due to Bogomolov and Mumford \cite{mori2006uniruledness}, or see \cite[Theorem 13.1.1]{huybrechts2016lectures}, says that a general polarized K3 surface $(S,B)$ of genus $g\ge 2$ contains a nodal integral rational curve $\G\subset S$ such that $\G\in |B|.$ This says in particular that $|B|$ is base-point-free \cite[Proposition 2.3.5]{huybrechts2016lectures}. Letting $g=2,$ then the polarization gives rise to a finite double cover $\phi_B\colon S\to\P^2.$ Now, note that $\G$ contains $2$ singular points: the genus formula implies that $p_a(\G)=2,$ then, using that $\G$ is rational, we deduce that the number of nodes is $\delta=p_a(\G)-p_g(\G)=2.$ Taking a nodal point $x\in\G,$ we have 
		\[
		\mult_x(\G)=2=2g-2=B^2=B\cdot\G,
		\]
		saying that $\eps(B;x)=1.$ Therefore $\G$ is a $1$-Seshadri curve for $B$ at $x$ having $x$ as point of maximum multiplicity, $\G$ being nodal, which is mapped finitely onto a line (by construction), just as said in Lemma \ref{lem:7.1}(iii). However, if $(S,B)$ is chosen very general, we can suppose $\rho(S)=1,$ which means that $\Pic(S)\cong\Z$ \cite[Proposition 1.2.4]{huybrechts2016lectures}. In particular there cannot exist $B$-lines as $1$-Seshadri curves.
		
		It's interesting to observe that, despite $\mult_x(\G)=2$ and $x\in\Ram(\phi),$ the exceptional divisor $E$ of the blow-up $\mu\colon\w S\to S$ at $x$ is not contained in the base scheme of $\mu^\ast B-E.$
	\end{rmk}
	
	To conclude this part, we see an example of $1$-Seshadri curve at $\phi_B^{-1}(\phi_B^{\hspace{0.000001mm}}(x)).$
	
	\begin{rmk}
		Let $\phi\colon S\to\P^2$ be a finite cover of degree $q$ branched over a smooth irreducible plane curve $D\subset\P^2$ and let $B=\phi^\ast\O_{\P^2}(1)$. Fix a point $x\in S$ such that $\phi(x)\notin D$ and set $\phi^{-1}(\phi(x))=\{x_1,\dots,x_q\}.$ The pullback via $\phi$ of a line $L\subset\P^2$ passing through $\phi(x)$ yields a divisor $C\in |B|$ containing all $x_i$'s. Combining Lemma \ref{lem:7.1-2} and Proposition \ref{prop:seshadri-multipoint}(1) we immediately see
		\[
		1\leq\frac{B\cdot C}{\sum_{j=1}^q\mult_{x_j}(C)}\leq\frac{B^2}{q}=1,
		\] which says that $C$ is a $1$-Seshadri curve for $B$ at $\phi^{-1}(\phi(x)).$
	\end{rmk}

	\section{Proof of Theorem \ref{thm:main} and some remarks}\label{sect:main}
	The proof of Theorem \ref{thm:main} follows very closely the proofs of \cite[Theorem 7.2 \& Theorem 1]{lopez2022positivity}. The principal obstruction is that a base-point-free ample divisor $B$ on a smooth projective variety $X$ generally does not separate points. Therefore the line bundle $\wB=\O_{\wX}(\mu^\ast B-E)$ on the blow-up $\mu\colon\wX\to X$ of a point with exceptional divisor $E,$ even if it is ample, can no longer be globally generated. Consequently a vector bundle on $\wX$ which is $0$-regular with respect to $\wB$ could not be generated by global sections. It comes necessary proving a version of Castelnuovo-Mumford regularity theorem which allows the polarization to have a non-empty base locus.
	
	\begin{prop}[Castelnuovo-Mumford] \label{prop:castelnuovo-mumford}
		Let $X$ be a projective variety of dimension $n\geq1$ and let $A$ be an ample line bundle on $X.$ Let $Y=\mathrm{Bs}(|A|)$ be the base scheme of $|A|$ with base ideal $\J=\mathfrak{b}(|A|).$ Suppose that $\dim Y=0$ and let $\mu\colon\wX\to X$ be the blow-up of $X$ along $Y$ with exceptional divisor $E.$ Let $\F$ be a vector bundle on $X$ which is $m$-regular with respect to $A.$ Then for every $k\geq\l\geq0$:
		\begin{itemize}\setlength\itemsep{0em}
			\item[(1)] Assuming that $\mu^\ast A-E$ is ample on $\wX,$ then $\F((m+k)A)$ is generated by global sections.
			\item[(2)] The multiplication map
			\[
			\mu_{m,k}\colon H^0(X,\F(mA))\o H^0(X,kA)\rightarrow H^0(X,\F((m+k)A)\o\J^{\o k})
			\]
			and the map
			\[
			\mu'_{m,k}\colon H^0(X,\F((m+k)A)\o\J^{\o k})\rightarrow H^0(X,\F((m+k)A)\o\J^k)
			\]
			are both surjective.
			\item[(3)] $\F\o\J^{\o\l}$ and $\F\o\J^\l$ are $(m+k)$-regular with respect to $A.$ 
		\end{itemize}
	\end{prop}
	\begin{proof}
		Letting $V=H^0(X,A)$ and $U=X-Y,$ the evaluation map $V_X=V\otimes\O_X\rightarrow A$ is surjective on $U.$ Following the construction in \cite[§B.2]{lazarsfeld2017positivity}, we can form the Koszul complex
		\[
		0\rightarrow\Lambda^rV_X(-rA)\rightarrow\cdots\rightarrow\Lambda^2V_X(-2A)\rightarrow V_X(-A)\xrightarrow{\eps}\J\rightarrow0,
		\]
		where $r=\dim V$ which is exact off $Y$ and with $\eps$ being surjective. Tensoring through by $\F(sA)\o\J^{\o t},$ with $t\geq0,$ we obtain the following complex 
		\begin{equation}\label{eq:koszul}
			0\rightarrow\Lambda^rV_X\o\F((s-r)A)\o\J^{\o t}\rightarrow\cdots\rightarrow V_X\o\F((s-1)A)\o\J^{\o t}\xrightarrow{\delta}\F(sA)\o\J^{\o (t+1)}\rightarrow0.
		\end{equation} 
		This remains exact off $Y,$ since $(\F(sA)\o\J^{\o t})_{|U}\cong\F(sA)_{|U}$ is locally free, and $\delta$ is still surjective because tensor product is right exact. For every $1\leq i\leq r,$ let $k_i=\dim \Lambda^iV.$
		
		\emph{Step 1: we prove (3).}\\
		The $(m+k)$-regularity of $\F\o\J^\l$ will follow once we know the $(m+k)$-regularity of $\F\o\J^{\o\l}.$ Indeed, consider the short exact sequence
		\[\begin{tikzcd}
			0\rar&\mc{K}\rar&\J^{\o q}\rar&\J^q\rar&0
		\end{tikzcd}\]
		for $q\geq1.$ For any $x\in U,$ the map $\O_{X,x}\cong\J^{\o q}_x\rightarrow\J^q_x\cong\O_{X,x}$ is an isomorphism, therefore $\mathrm{supp}(\mc{K})=Y$ has dimension $0.$ As a consequence, $H^i(X,\mc{K}\o\F(pA))=0$ for every $i>0,$ forcing the isomorphism $H^i(X,\F(pA)\o\J^{\o q})\cong H^i(X,\F(pA)\o\J^{q})$ for all $i>0,p\in\Z$ and $q\geq1$ which in turn gives the claim. In addition to this, we observe that it implies the surjectivity of the map 
		\begin{equation}\label{eq:mult-map-product-ideal}
			H^0(X,\F(pA)\o\J^{\o q})\rightarrow H^0(X,\F(pA)\o\J^q)
		\end{equation}
		for all $p\in\Z.$ 
		
		To prove the $(m+k)$-regularity of $\F\o\J^{\o\l}$ we proceed by induction $k\geq\l\geq0$ with base case $k=\l=0$ given by the hypothesis. Consider separately the case $\l=0$ with $k\geq1,$ and set $(s,t)=(m+k-i,0)$ in (\ref{eq:koszul}) for $i>0.$ By induction we know that $\F$ is $(m+k-1)$-regular. Hence, we get the vanishing 
		\[
		H^{i+j}(X,\Lambda^{j+1}V_X\o\F((m+k-i-j-1)A))\cong H^{i+j}(X,\F(((m+k-1)-(i+j))A))^{\oplus k_{j+1}}=0
		\] for $0\leq j\leq r-1.$ Applying \cite[Proposition B.1.2]{lazarsfeld2017positivity} we immediately deduce that $H^i(X,\F((m+k-i)A)\o\J)=0$ for $i>0,$ which amounts to say that $\F\o\J$ is $(m+k)$-regular. Twist the short exact sequence 
		\[
		\begin{tikzcd}
			0\rar&\J\rar&\O_X\rar&\O_Y\rar&0
		\end{tikzcd}
		\]
		through by $\F((m+k)A)$ and then take the associated long exact sequence in cohomology. One obtain short exact sequences
		\[
		\begin{tikzcd}
			0=H^i(X\F((m+k)A)\o\J)\rar&H^i(X,\F((m+k)A))\rar&H^i(Y,\F((m+k)A)_{|Y})=0
		\end{tikzcd}
		\] for every $i>0.$ This says that $\F$ is $(m+k)$-regular. 
		
		Now, let $k\geq\l\geq1,$ fix $i>0$ and set $(s,t)=(m+k-i,\l-1)$ in (\ref{eq:koszul}). For every $0\leq j\leq r-1$ we have the vanishing 
		\begin{align*}
			\quad&H^{i+j}(X,\Lambda^{j+1}V_X\o\F((m+k-i-j-1)A)\o\J^{\o (\l-1)})\\
			&\cong H^{i+j}(X,\F(((m+k-1)-(i+j))A)\o\J^{\o (\l-1)})^{\oplus k_{j+1}}\\
			&=0
		\end{align*}
		due to the $(m+k-1)$-regularity of $\F\o\J^{\o(\l-1)}$ given by the inductive hypothesis. Using again  \cite[Proposition B.1.2]{lazarsfeld2017positivity} we get $H^i(X,\F((m+k-i)A)\o\J^{\o\l})=0,$ completing the inductive step.
		
		\emph{Step 2: $\F((m+k)A)$ is generated by global sections at every point of $Y.$}\\
		Set $(s,t)=(m+k,0)$ in (\ref{eq:koszul}). The $(m+k)$-regularity of $\F$ given by Step 1 implies the vanishing
		\[
		H^{1+j}(X,\Lambda^{j+1}V_X\o\F((m+k-1-j)A))\cong H^{1+j}(X,\F(((m+k)-(1+j))A))^{\oplus k_{j+1}}=0
		\] for each $0\leq j\leq r.$ It follows from \cite[Proposition B.1.2]{lazarsfeld2017positivity} that $H^1(X,\F\left((m+k)A\right)\o\J)=0.$ Taking the cohomology of the short exact sequence
		\[
		\begin{tikzcd}
			0\rar& \F((m+k)A)\o\J\rar&\F((m+k)A)\rar&\F((m+k)A)\o\O_Y\rar&0,
		\end{tikzcd}
		\] we see that this vanishing forces the restriction map $$H^0(X,\F((m+k)A))\rightarrow H^0(X,\F((m+k)A)\o\O_Y)$$ to be surjective. This proves the claim.
		
		\emph{Step 3: we prove (2)}\\
		Proceed by induction on $k\geq0$ with trivial base case $k=0.$
		Assume $k\geq1$ and set $(s,t)=(m+k,k-1)$ in (\ref{eq:koszul}). The coherent sheaf $\F\o\J^{\o(k-1)}$ is $(m+k-1)$-regular by Step 1, hence
		\[
		H^j(X,\Lambda^{1+j}V_X\o\F((m+k-1-j)A)\o\J^{\o(k-1)})=0
		\] for every $j\geq1.$ Applying again \cite[Example B.1.3]{lazarsfeld2017positivity} we obtain the surjectivity of the map
		\[
		H^0(X,\F((m+k-1)A)\o\J^{\o(k-1)})\o H^0(X,A)\rightarrow H^0(X,\F((m+k)A)\o\J^{\o k}).
		\] Applying the inductive hypothesis wet get a surjective map
		\[
		H^0(X,\F(mA))\o H^0(X,(k-1)A)\o H^0(X,A)\rightarrow H^0(X,\F((m+k)A)\o\J^{\o k}).
		\] 	Factoring through the natural multiplication map $H^0(X,(k-1)A)\o H^0(X,A)\rightarrow H^0(X,kA),$ we obtain a commutative diagram 
		\[
		\begin{tikzcd}
			H^0(X,\F(mA))\o H^0(X,(k-1)A)\o H^0(X,A)\arrow[rrr]\arrow[dd] &&& H^0(X,\F((m+k)A)\o\J^{\o k}) \\
			&&&\\
			H^0(X,\F(mA))\o H^0(X,kA) \arrow[uurrr] &
		\end{tikzcd}
		\]
		which forces the surjectivity of
		\[
		\mu_{m,k}\colon H^0(X,\F(mA))\o H^0(X,kA)\rightarrow H^0(X,\F((m+k)A)\o\J^{\o k}).
		\] 		The claim for $\mu'_{m,k}$ in (2) is finally obtained using the surjectivity of (\ref{eq:mult-map-product-ideal}).
		
		\emph{Step 4: $\F(mA)$ is generated by global sections on $X-Y$ assuming the hypothesis in (1).}\\
		The line bundle $\mu^\ast A-E$ is ample on $\wX,$ therefore the vector bundle 
		\[
		\mu^\ast(\F((m+k)A))\o\O_{\w{X}}(-kE)\cong \mu^\ast(\F(mA))\o\O_{\w{X}}(\mu^\ast A-E)^{\o k}
		\]	
		is generated by global sections for every $k\gg0.$ This amounts to say that the morphism
		\[
		\mathrm{ev}_k\colon H^0(\w{X},\mu^\ast(\F(mA))\o\O_{\w{X}}(\mu^\ast A-E)^{\o k})\o\O_{\w{X}}\rightarrow \mu^\ast(\F(mA))\o\O_{\w{X}}(\mu^\ast A-E)^{\o k}
		\]
		is surjective for every $k\gg0.$ Moreover, for $k\gg0$ we also have the isomorphism 
		\[
		H^0(\w{X},\mu^\ast(\F((m+k)A))\o\O_{\w{X}}(-kE))\cong H^0(X,\F((m+k)A)\o\J^k)
		\]
		obtained by combining \cite[Lemma 5.4.24]{lazarsfeld2017positivity} and projection formula. Fix $k\gg0$ such that both properties hold. Since $\mu_U=\mu_{|\mu^{-1}(U)}\colon\mu^{-1}(U)\rightarrow U$ is an isomorphism, the functor $(\mu_U)_\ast$ preserves surjectivity. Therefore, we can form the following commutative diagram on $U$:
		\[
		\begin{tikzcd}
			H^0(X,\F(mA))\o (kA)_{|U}\arrow[rr] &&  \F((m+k)A)_{|U}\cong\left(\F((m+k)A)\o\J^k\right)_{|U}\\ 
			&&\\
			H^0(X,\F(mA))\o H^0(X,kA)\o\O_U\arrow[uu]\arrow[rr, "\mu_{m,k}\o\mathrm{id}_{\O_U}"]&& H^0(X,\F((m+k)A)\o\J^k)\o\O_U.\arrow[uu, "(\mu_U)_\ast(\mathrm{ev}_k)"]
		\end{tikzcd}
		\]
		The map $\mu_{m,k}\o\mathrm{id}_{\O_U}$ is surjective by Step 3, as well as $(\mu_U)_\ast(\mathrm{ev}_k)$ by the above discussion. Then the commutativity forces the map 
		\[
		H^0(X,\F(mA))\o(kA)_{|U}\rightarrow \F((m+k)A)_{|U}
		\] 
		to be surjective. Twisting by the sheaf $((-k)A)_{|U},$ we finally obtain the claim.

		\emph{Step 5: we prove (1)}\\
		Combining Steps 2 and 4 we get the case $k=0.$ The remaining cases are obtained in the same way using that $\F$ is $(m+k)$-regular for every $k\geq1$ by Step 1.
	\end{proof}
	
	We are now ready to prove Theorem \ref{thm:main}. 
	
	\begin{proof}[Proof of Theorem \ref{thm:main}]
		Write $\phi=\phi_B$ and let $\mu\colon\wX\to X$ be the blow-up at $x$ with exceptional divisor $E.$ Set $\wB=\mu^\ast B-E$ and let $Y\subset\wX$ be its base scheme. 
		
		Suppose $x$ satisfies (a) and (b). We show that the pair $(\wX,\wB)$ satisfies the hypotheses in Proposition \ref{prop:castelnuovo-mumford}. First of all we observe that $\wB$ is ample on $\wX$: $\phi$ is unramified at $x,$ so the schematic fibre over $\phi(x)$ is the disjoint union $\phi^{-1}(\phi(x))=\{x\}\bigsqcup Z$ with $x$ being a smooth point for $\phi^{-1}(\phi(x)),$ 
		whence  Proposition \ref{prop:seshadri-multipoint}(3) yields $\eps(B;x)\geq\eps(B;\phi^{-1}(\phi(x)))>1.$ Thus Lemma \ref{lem:7.1} give the assertion. Moreover, 
		the base scheme $Y$ is the strict transform of $Z$ under $\mu,$ in particular it is $0$-dimensional by the finiteness of $\phi.$ Now, consider the blow-up $\rho\colon X'\to \wX$ of $\wX$ along $Y$ with exceptional divisor $E_Y.$ Then the line bundle $\rho^\ast\wB-E_Y$ is base-point-free by construction. On the other hand, the composition $\pi=\rho\circ\mu\colon X'\to X$ is the blow-up of $X$ along $\phi^{-1}(\phi(x))$ and its exceptional divisor is $\b E=E'+E_Y,$ where $E'$ is the strict transform of $E$ via $\rho.$ If we show that $\rho^\ast \wB-E_Y=\pi^\ast B-\b E$ is also strictly nef, then it will be ample by \cite[Corollary 1.2.15]{lazarsfeld2017positivity}, completing the claim. For, let $\b C\subset X'$ be an irreducible curve. If $\b E\cdot\b C<0,$ then $\b C\subset\b E$ and it is contracted by $\pi$ to a point. Therefore we have
		\[
		(\pi^\ast B-\b E)\cdot\b C=-\b E\cdot \b C>0.
		\] Now suppose $\b E\cdot\b C\geq0.$ As $B$ is ample, by Lemma \ref{lem:seshadri-subscheme2} we can find $\b\s>1$ such that $\pi^\ast B-\b\s\b E$ is ample. Then we get
		\[
		(\pi^\ast B-\b E)\cdot\b C=(\pi^\ast B-\b\s\b E)\cdot\b C+(\b\s-1)\b E\cdot\b C>0,
		\] hence proving the strictly nefness.
		
		Now, for every integer $0\leq s\leq n-1,$ we have $R^j\mu_\ast\O_{\w{X}}(sE)=0$ for $j>0$ and $\mu_\ast\O_{\w{X}}(sE)=\O_{X},$ see e.g. \cite[Proof of Lemma 4.1]{bertram1991vanishing}. Combining $0$-regularity with respect to $B,$ projection formula \cite[Exercise III.8.3]{hartshorne2013algebraic} and Leray spectral sequence \cite[Exercise III.8.1]{hartshorne2013algebraic}, we obtain 
		\[
		H^i(\w{X},(\mu^\ast\E\o\O_{\w{X}}(-E))(-i\w{B}))=H^i(\w{X},\mu^\ast(\E(-iB))\o\O_{\w{X}}((i-1)E))\cong H^i(X,\E(-iB))=0
		\] for every $i>0.$ So $\w{\E}=\mu^\ast\E\o\O_{\w{X}}(-E)$ is $0$-regular with respect to the ample divisor $\w{B}.$ Since $\rho^\ast\wB(-E_Y)$ is ample, it follows from Proposition \ref{prop:castelnuovo-mumford}(i) that $\w{\E}$ is generated by global sections. Then $c_1(\w{\E})=\mu^\ast c_1(\E)-rE$ is base-point-free, so nef. By Kleiman's theorem \cite[Theorem 1.4.9]{lazarsfeld2017positivity} and \cite[Lemma 5.1.10]{lazarsfeld2017positivity}, we deduce that
		\[
		0\leq (\mu^\ast(\E)-rE)^k\cdot\w{Z}=c_1(\E)^k\cdot Z+r^k(-1)^k(-1)^{k+1}\mult_x(Z)=c_1(\E)^k\cdot Z-r^k\mult_x(Z),
		\] where $\w Z\subset \w X$ is the strict transform of a subvariety $Z\subset X$ of dimension $k\geq1$ passing through $x.$ This proves (\ref{eq:ulrich-c1}).
		
		To prove the last part of the statement, suppose $X$ is not generically covered by $1$-Seshadri curves for $\phi.$ By Remark \ref{rmk:covered-1-seshadri}, we can find a point $x\in X$ such that $\phi^{-1}(\phi(x))$ is smooth, which means that $x\notin\Ram(\phi),$ having $\eps(B;\phi^{-1}(\phi(x)))>1.$ 
		Then, choosing $Z=X$ in (\ref{eq:ulrich-c1}), we immediately get (\ref{eq:ulrich-bigness}). The vector bundle $\w\E=(\mu^\ast\E)(-E)$ is generated by global sections, hence nef. Then Remark \ref{rmk:seshadri-vector}(b) tells that $\eps(\E;x)\geq1,$ which means $x\notin\B_+(\E)$ by Remark \ref{rmk:augmented-vector}. By Theorem \ref{thm:V-big}, $\E$ is V-big as required. To conclude, suppose that $\E$ is $B$-Ulrich of rank $r\geq 2.$ In particular it is $0$-regular by Lemma \ref{lem:claim}(i). Note that due to (b), $\O_X$ cannot be $B$-Ulrich:  if this was true, then we would have $(X,B)\cong(\P^n,\OPn(1))$ (Lemma \ref{lem:claim}(ix)) leading to a contradiction since $\eps(\P^n,\OPn(1),1)=1.$ Consequently, a decomposition as $\E\cong\O_X^{\oplus(r-1)}\oplus\det(\E)$ is not admissible \cite[Remark 2.4.2]{beauville2018introduction}. Therefore \cite[Theorem 1]{sierra2009degree} and Lemma \ref{lem:claim}(ii) imply that $c_1(\E)^n\geq h^0(X,\E)-r=r(d-1),$ completing the proof.
	\end{proof}
	
	Observe that the hypothesis of Theorem \ref{thm:main} can be equivalently stated as follows.
	
	\begin{rmk}\label{rmk:equivalent-conditions-7.2}
		Let $x\in X$ be a point in a smooth projective variety and let $B$ a globally generated ample line bundle on $X.$ Consider the following conditions:
		\begin{itemize}
			\item[(a)] $x\notin\Ram(\phi_B).$
			\item[(b)] $\eps(X,B;\phi_B^{-1}(\phi_B^{\hspace{0.0000001mm}}(x)))>1.$
			\item[(c)] $\eps(X,B;x)>1.$
			\item[(d)] $\eps(\wX,\mu^\ast B-E;\Bs(|\mu^\ast B-E|))>1,$ where $\mu\colon \wX\to X$ is the blow-up at $x$ with exceptional divisor $E.$
		\end{itemize}
		Then (a)+(b) is equivalent to (a)+(c)+(d).
		
		In order to see this, let's first set $\phi=\phi_B,$ $\wB=\mu^\ast B-E$ and $Y=\Bs(|\w B|).$ Henceforth we suppose that (a) holds. This says that $Y$ is the strict transform under $\mu$ of the closed subscheme $Z\subset\phi^{-1}(\phi(x))\subset X,$ which does not contain $x$ and such that $\phi^{-1}(\phi(x))=\{x\}\bigsqcup Z.$  Let $\rho\colon X'\to X$ be the blow-up of $\wX$ with exceptional divisor $E_Y.$ Set $E'$ to be the strict transform of $E$ via $\rho.$ Then the composition $\pi=\rho\circ\mu\colon X'\to X$ is the blow-up of $X$ along $\phi^{-1}(\phi(x))$ with exceptional divisor $\b E=E'+E_Y.$ In particular we have $\rho^\ast\wB-E_Y=\pi^\ast B-\b E.$ Recall that $\O_{\b E}(-\b E)$ is ample since the centre of the blow-up $\pi$ is $0$-dimensional (Remark \ref{rmk:exceptional}). 
		
		Suppose (c) and (d) hold. Then (c) says that $\wB$ is ample (Lemma \ref{lem:7.1}). We claim that $\rho^\ast \wB-E_Y$ is strictly nef. Let $C'\subset X'$ be an irreducible curve. If $E_Y\cdot C'<0,$ then $C'\subset E_Y,$ in particular it is contracted by $\rho$ to a point, whence
		\[
		(\rho^\ast \wB-E_Y)\cdot C'=-E_Y\cdot C'>0.
		\] Suppose $E_Y\cdot C'\geq0.$ Since $\eps(\wX,\wB;Y)>1,$ by Lemma \ref{lem:seshadri-subscheme2} we can find $\s>1$ such that $\rho^\ast\wB-\s E_Y$ is ample. (Here we are using the ampleness of $\wB.$)  Then we get
		\[
		(\rho^\ast\wB-E_Y)\cdot C'=(\rho^\ast\wB-\s E_Y)\cdot C'+(\s-1)E_Y\cdot C'>0.
		\] This proves the claim. Then it follows by \cite[Corollary 1.2.15]{lazarsfeld2017positivity} that $\rho^\ast\wB-E_Y=\pi^\ast B-\b E$ is ample. Thanks to the openness of the ample cone, this implies that $\eps(X,B;\phi^{-1}(\phi(x)))>1$ (see also the proof of Lemma \ref{lem:7.1-2}), that is (b).
		
		Conversely, assume that (b) holds. Then (c) follows from (b) (and (a)) by Proposition \ref{prop:seshadri-multipoint}(3). If we show that $\pi^\ast B-\b E=\rho^\ast\wB-E_Y$ is ample, then $\eps(\wX,\wB;Y)>1$ (see the proof of Lemma \ref{lem:7.1-2}). Therefore (d) will follow once we have proved the ampleness of $\pi^\ast B-\b E=\rho^\ast\wB-E_Y.$ As it is globally generated, by \cite[Corollary 1.2.15]{lazarsfeld2017positivity} we only need to show it is strictly nef. Let $\b C\subset X'$ be an irreducible curve. If $\b E\cdot\b C<0,$ then $\b C\subset\b E$ and it is contracted to a point via $\pi.$ Therefore we have
		\[
		(\pi^\ast B-\b E)\cdot\b C=-\b E\cdot \b C>0.
		\] Now suppose $\b E\cdot\b C\geq0.$ As $B$ is ample, by Lemma \ref{lem:seshadri-subscheme2} we can find $\b\s>1$ such that $\pi^\ast B-\b\s\b E$ is ample. Then we get
		\[
		(\pi^\ast B-\b E)\cdot\b C=(\pi^\ast B-\b\s\b E)\cdot\b C+(\b\s-1)\b E\cdot\b C>0,
		\] completing the proof.
	\end{rmk}
	
	This observation leads to the following geometrical interpretation.
	
	\begin{rmk}
		Let $x\in X$ be a point in a smooth projective variety and let $B$ a globally generated ample line bundle on $X.$ Consider conditions (a)-(b)-(c)-(d) in Theorem \ref{thm:main} and Remark \ref{rmk:equivalent-conditions-7.2}. In light of Lemmas \ref{lem:7.1} - \ref{lem:7.1-2} and of Remark \ref{rmk:seshadri-B-line}, (a)+(c) is equivalent to say that there is no $B$-line passing through $x.$
	\end{rmk}
	
	\begin{rmk}
		Let $H$ be a very ample line bundle on a smooth projective variety $X.$ The morphism $\phi=\phi_H$ is an embedding,  hence $\phi^{-1}(\phi(x))=\{x\}$ is smooth. Then, regarding Theorem \ref{thm:main}: (a) is always satisfied and (b) is equivalent to say that there is no line passing through $x$ (Lemmas \ref{lem:7.1} - \ref{lem:7.1-2} and Remark \ref{rmk:seshadri-B-line}). Thus the statement of Theorem \ref{thm:main} reduces to \cite[Theorem 7.2 \& Theorem 1]{lopez2022positivity}. 
	\end{rmk}
	
	\begin{rmk}\label{rmk:necessity-main}
		Let $X$ be a smooth projective variety of dimension $n\geq 1$ and let $B$ an ample line bundle which is generated by global sections. In order to satisfy (b) in Theorem \ref{thm:main}, the linear series $|B|$ cannot induce a finite morphism $\phi=\phi_B\colon X\rightarrow\P^n$ onto the projective space. 
		
		Indeed, if this happens, projection formula implies $d:=B^n=\deg(\phi)\cdot\deg(\P^n)=\deg(\phi).$ Then, if $\phi(x)$ is not a branch point, the fibre over $\phi(x)$ consists of $\deg(\phi)=d$ distinct smooth points. 
		Then (\ref{eq:seshadri-submaximal}) says that
		\[
		\eps(X,B;\phi^{-1}(\phi(x)))\leq\left(\frac{B^n}{d}\right)^\frac{1}{n}=1,
		\] forcing $\eps(X,B;\phi^{-1}(\phi(x)))=1$ (Lemma \ref{lem:7.1-2}).
	\end{rmk}
	
	Let's see an example.
	
	\begin{ex}\label{ex:abelian}
		Let $(S,B)$ be a polarized abelian surface of type $(2,2)$ which we assume to be non-isomorphic to the product of two elliptic curves. In virtue of \cite[Exercise 8.11(1)]{birkenhake2004complex}, this is the general case. According to \cite[§10.1, p. 282]{birkenhake2004complex}, the line bundle $B$ is globally generated and $B=L^{\o 2}$ for an ample line bundle $L.$ As $S$ does not split, $B$ induces a morphism $\phi_B\colon S\to K\subset\P^3$ of degree $2$ onto its image $K,$ which is a Kummer surface.
		
		Once clarified the setting, we can observe that the argument in \cite[Theorem 1]{beauville2016ulrich} works also in the case of an ample and globally generated polarization. Therefore $S$ supports a $B$-Ulrich vector bundle $\E$ of rank $2.$ Since $\eps(L)\geq 4/3$ \cite[Theorem 6.4.4(a)]{bauer2009primer}, by (\ref{eq:seshadri-lower}) and Proposition \ref{prop:seshadri-properties}(1) we obtain
		\[
		\eps(B;x_1,x_2)\geq\frac{1}{2}\eps(B)=\eps(L)\ge \frac{4}{3}>1
		\]
		for every pair of distinct points $x_1,x_2\in S.$ This holds in particular for all pairs of points lying on the fibre over $\phi_B(x)$ with $x\notin\Ram(\phi_B).$ Therefore $\E$ is $V$-big by Theorem \ref{thm:main}.
	\end{ex}

	The conditions in Theorem \ref{thm:main} are easy to handle when the polarization defines a birational morphism.
	
	\begin{rmk}
		Let $(X,B)$ is a smooth polarized projective variety with $B$ generated by global sections such that $\phi_B\colon X\rightarrow\P^N$ is birational onto its image. If $X$ is not covered by $1$-Seshadri curves for $B,$ then every vector bundle which is $0$-regular with respect to $B$ is V-big.

		Indeed, given a point $x\in X$ not belonging to any $1$-Seshadri curve for $B$ at $x,$ that is $\eps(B;x)>1$ (Lemma \ref{lem:7.1}), by \cite[Lemma 1.4]{ein1995local} we can find a dense open subset $U\subset X$ where $\phi_B$ is an isomorphism and $\eps(B;y)>1$ for every $y\in U.$ Then the assertion follows from Theorem \ref{thm:main}.
	\end{rmk}
	
	\begin{ex}
		Let $(S,B)$ be polarized abelian surface of type $(1,4)$ which does not split as the product of two elliptic curves. Then \cite[Exercise 8.11(1)]{birkenhake2004complex} says that this is the general case. Then $\phi_B\colon S\rightarrow\b S\subset\P^3$ is a morphism which is birational onto a singular octic surface $\b S$ \cite[§10.5, p. 302, lines 19-24]{birkenhake2004complex}. We also know that $S$ supports a $B$-Ulrich vector bundle $\E$ (see Example \ref{ex:abelian}). As $\eps(B;1)=\eps(B)\ge 4/3$ by \cite[Theorem 6.4.4(a)]{bauer2009primer}, $\E$ is V-big by the above Remark.
	\end{ex}
	
	The conditions in Theorem \ref{thm:main} are not necessary for the bigness of an Ulrich bundle. Combining \cite[Proof of Theorem 1 \& Remark 2.2]{lopez2021classification}, we get the following characterization of big $B$-Ulrich vector bundles on surfaces.
	
	\begin{lemma}\label{rmk:big-ulrich-surfaces}
		Let $(S,B)$ a smooth projective surface together with a base-point-free ample divisor and let $\E$ be a $B$-Ulrich vector bundle. Then $\E$ is big if and only if $c_1(\E)^2>0.$
	\end{lemma}
	\begin{proof}
		As $\E$ is globally generated (Lemma \ref{lem:claim}(i)), this is clear if $\rk(\E)=1.$ Thus we  can assume $r=\rk(\E)\geq2.$ Necessity follows by \cite[Remark 2.2]{lopez2021classification}. For sufficiency, we know by hypothesis that $c_1(\E)$ is big and nef. Hence $B^2>1,$ for otherwise Lemma \ref{lem:degree-1} would imply $(S,B,\E)\cong(\P^2,\O_{\P^2}(1), \O_{\P^2}^{\oplus r})$ giving a contradiction as the $c_1$ would not be big. In particular, by Lemma \ref{lem:claim}(ii), $h^0(X,\E)=r\cdot B^2\geq r+2,$ and $H^1(X,\det(\E)^\ast)=0$ by \cite[Theorem 4.3.1]{lazarsfeld2017positivity}. Then \cite[Theorem 3.2]{bini2020big} implies that $\E$ is big.
	\end{proof}
	
	\begin{ex}
		Let $\phi=\phi_B\colon S\rightarrow \P^2$ be a finite double cover branched along a general smooth sextic with $\O_S(B)\cong\phi^\ast\O_{\P^2}(1).$ It easily follows, for instance by \cite[Lemmas 17.1-17.2]{barth2015compact}, that $(S,B)$ is a smooth K3 surface of genus $g=2.$  Moreover $S$ supports a special $B$-Ulrich vector bundle $\E$ of rank $2$ \cite[Theorem 1.2]{sebastian2022rank}. However by Remark \ref{rmk:necessity-main}, this pair $(S,B)$ cannot satisfy condition (b) of Theorem \ref{thm:main}.  
		On the other hand,  as $c_1(\E)^2=(3B)^2>0,$ $\E$ is big by Lemma \ref{rmk:big-ulrich-surfaces}.
	\end{ex}
	
	More generally, for minimal surfaces of non-negative Kodaira dimension we have the following.
	
	\begin{rmk}
		Let $S$ be a minimal smooth projective surface of Kodaira dimension $\kappa(S)\ge 0$ and let $B$ be a globally generated ample line bundle on $S.$ Then every special $B$-Ulrich bundle of rank $2$ is big. In particular, if $\Pic(S)\cong\Z,$ every $B$-Ulrich bundle of rank $2$ is big.
		
		Let's prove this. In case $\Pic(S)\cong\Z,$ we can deduce from Lemma \ref{lem:claim}(iii) that every $B$-Ulrich bundle $\F$ of rank $2$ has $c_1(\F)=K_S+3B,$ saying that it is special for $B.$ (See also \cite[Lemma 3.2]{lopez2022positivity}.) Thus we only need to show the first part. 
		Since $K_S$ is nef by hypothesis, the Chern class $c_1(\E)=K_S(3B)$ is ample for every special $B$-Ulrich bundle $\E$ of rank $2.$ Hence $c_1(\E)^2>0$ by Nakai-Moishezon-Kleiman criterion \cite[Theorem 1.4.9]{lazarsfeld2017positivity}, and the conclusion follows by Lemma \ref{rmk:big-ulrich-surfaces}.
	\end{rmk}
	
	In this slightly wider setting in which we consider Ulrich bundles with respect to ample and globally generated line bundles, there are more non-big Ulrich bundles. In fact, the following examples are not ascribable to the the list of non-big Ulrich bundles (with respect to a very ample divisor) on surfaces and threefolds in \cite[Theorems 1 \& 2]{lopez2021classification}.
	
	\begin{ex}\label{ex:DP-surface}
		Let $(S,-K_S)$ be a Del Pezzo surface of degree $d=2$ and let $\E=L(-K_S)$ be the $-K_S$-Ulrich line bundle of Proposition \ref{prop:ulrich-DP-surface}. Despite $\eps(-K_S;x)=4/3$ for a general point $x\in S$ \cite[Theorem 6.3.4]{bauer2009primer}, we cannot apply Theorem \ref{thm:main} because we are in the situation of Remark \ref{rmk:necessity-main}. Indeed, we have $c_1(\E)^2=0$ by construction. Thus $\E,$ which is already nef, is not big.
	\end{ex}
	
	\begin{ex}\label{ex:DP-threefold}
		Consider a Del Pezzo threefold $(X,B)$ of degree $d=2$ and let $\E$ be the special $B$-Ulrich bundle of rank $2$ constructed in Proposition \ref{prop:ulrich-DP-threefold}. We are in the situation of Remark \ref{rmk:necessity-main}, hence we cannot apply Theorem \ref{thm:main}. Due to $\E$ is globally generated (Lemma \ref{lem:claim}(i)), to show that $\E$ is non-big, it is enough to check that $s_3(\E^\ast)=0$ \cite[Remark 2.2]{lopez2021classification}. Using \cite[Examples 8.3.4-8.3.5]{lazarsfeld2017positivity2} and the fact that $\E$ is special of rank $2,$ we get
		\[
		s_3(\E^\ast)=s_{(1,1,1)}(\E)=c_1(\E)^3-2c_1(\E)\cdot c_2(\E)+c_3(\E)=(2B)^3-2(2B)\cdot c_2(\E)=16-4c_2(\E)\cdot B.
		\] To compute the last intersection product, we use Lemma \ref{lem:claim}(iv):
		\begin{align*}
			c_2(\E) \cdot B &=\frac{1}{2}\left[c_1(\E)^2-c_1(\E)\cdot K_X\right]\cdot B+\frac{2}{12}\left[K_X^2+c_2(X)-\frac{44}{2}B^2\right]\cdot B\\
			&=\frac{1}{2}\left[(2B)^2\cdot B-(2B)\cdot(-2B)\cdot B\right]+\frac{1}{6}\left[(-2B)^2\cdot B+c_2(X)\cdot B-22(B)^3\right]\\
			&=8+\frac{1}{6}\left[-36+c_2(X)\cdot B\right].
		\end{align*}
		Applying  Hirzebruch-Riemann-Roch theorem on $\O_X,$ one finds that $c_2(X)\cdot B=12.$ 		In conclusion, we obtain
		\[
		s_3(\E^\ast)=16-4\left[8+\frac{1}{6}\left(12-36\right)\right]=16-4\cdot 4=0,
		\] giving the assertion.
	\end{ex}
	
	\section{Augmented base locus of an Ulrich vector bundle}
	The main goal of this section is characterizing the augmented base locus of a $B$-Ulrich bundle (Theorem \ref{thm:ulrich-augmented}). As said in the Introduction, we will need an additional hypothesis on our globally generated polarization $B$: we will suppose that $|B|$ admits a linear subsystem $|V|\subset |B|$ inducing a morphism which is étale onto the schematic image. In light of the previous section the assumption of being unramified is reasonable, in fact it turns out to be fundamental: in Remarks \ref{rmk:DP-surface} - \ref{rmk:DP-threefold} we show that Theorem \ref{thm:ulrich-augmented} and Corollary \ref{cor:ulrich-augmented} no longer apply when $\phi_V$ is flat but not unramified (on the schematic image).\\
	
	A $B$-Ulrich bundle is globally generated, hence nef. In particular the stable and the restricted base loci are empty \cite[Proposition 5.2]{bauer2015positivity}. In view of Remark \ref{rmk:augmented-vector}, to determine its augmented base locus, it suffices to characterize the points where the Seshadri constant vanishes.
	
	We begin with a result of general nature.
	
	\begin{lemma}\label{lem:step1}
		Let $X$ be a smooth projective variety and let $\E$ be a globally generated vector bundle on $X.$ Suppose there is an integral rational curve $C\subset X$ such that $\E_{|C}$ is not ample. Then $\eps(\E;x)=0$ for every $x\in C.$
	\end{lemma}
	\begin{proof}
		Let $\nu\colon C'\to C$ be the normalization of $C$ and let $f=\iota\circ\nu\colon C'\to X$ be the composition of $\nu$ with the inclusion $\iota\colon C\hookrightarrow X.$ Since $C'\simeq\P^1,$ the vector bundle $f^\ast\E=\nu^\ast\E_{|C}$ has rank $r=\rk(\E)$ and splits as 
		\[
		f^\ast\E=\O_{\P^1}(a_1)\oplus\cdots\oplus\O_{\P^1}(a_r),
		\] where $a_1\geq\cdots\geq a_r\geq0$ since $f^\ast\E$ is globally generated. However, $f^\ast\E=\nu^\ast\E_{|C}$ cannot be ample for otherwise $\E_{|C}$ would be so \cite[Proposition 6.1.8(iii)]{lazarsfeld2017positivity2}. In particular, we have (at least) $a_r=0$ \cite[Example 6.1.3]{lazarsfeld2017positivity2}. Then \cite[Lemma 3.31]{fulger2021seshadri} implies that  $\eps(f^\ast\E;y)=0$ for every $y\in C'.$ Consequently $$\eps(\E_{|C};x)=\frac{\eps(\nu^\ast\E_{|C};y)}{\mult_x(C)}=0$$ for any $y\in C'$ and every $x\in C$ (see \cite[p. 12, lines 5-6]{fulger2021seshadri}). For any $z\in X,$ we have $$\eps(\E;z)=\inf_{z\in F\subset X}\eps(\E_{|F};z)$$  where $F$ ranges over all irreducible curves in $X$ passing through $z$ (see, e.g., the proof of \cite[Corollary 3.21]{fulger2021seshadri}). Therefore, we obtain $0\leq\eps(\E;x)\leq\eps(\E_{|C};x)=0$ for every $x\in C,$ which gives the assertion.
	\end{proof}
	
	In this section we will use the following convention.
	
	\begin{notat} Given a projective scheme $X$ and a base-point-free linear series ${|V|}\subset |L|,$  let $\bX$ be the schematic image of $\phi_{V}\colon X\to\P^N,$ that is just the reduced induced subscheme structure on $\phi_{V}(X)\subset\P^N.$ Let $\iota\colon\bX\hookrightarrow\P^N$ denote the inclusion. By construction there is a morphism $\b{\phi_{V}}\colon X\to\bX$ such that
		$$\phi_{V}=\iota\circ\b{\phi_{V}}\colon X\to\bX\hookrightarrow\P^N,$$ which then satisfies $\b{\phi_{V}}\O_{\bX}(1)\cong L.$ Furthermore $\b{\phi_V}$ is finite when $L$ is ample: in this case $\phi_V$ is finite (see, e.g., Remark \ref{rmk:B-line}), hence $\b{\phi_V}$ must be finite as well being proper and quasi finite.	Henceforth we will use this fact without further reference and we will make no distinction between $\phi_{V}$ and $\b{\phi_{V}}.$
	\end{notat}
	
	\begin{rmk}
		The hypotheses of Theorems \ref{thm:ulrich-augmented} - \ref{thm:ulrich-ampleness} and Corollary \ref{cor:ulrich-augmented} are obviously satisfied by very ample linear series. However these are not the only ones:		certainly such line bundles must separate tangents at all point, or equivalently must give rise a local isomorphism, but one can construct several examples of globally generated ample line bundles with a linear system inducing an étale morphism (onto the schematic image) which is not a global isomorphism. Let's proceed as follows. Take a smooth projective variety $Y$ with either $b_1(Y)\neq 0$ or with $H_1(Y,\Z)$ containing $k$-torsion. In both cases we can find a smooth projective variety $X$ and a finite unramified morphism $\pi\colon X\to Y$ which is not an isomorphism \cite[Proposition 18.1]{barth2015compact}, but which is étale by \cite[Exercise III.9.3(a)]{hartshorne2013algebraic}. Then the triple $(X,B,V),$ where $B=\pi^\ast H$ for any very ample line bundle $H$ on $Y,$ and $V=\Imm\left(\pi^\ast\colon H^0(Y,H)\to H^0(X,\pi^\ast H)\right)\simeq H^0(Y,H)$ \cite[Corollaire 2.2.8]{grothendieck1965elements} satisfies those hypotheses since $\phi_V=\phi_H\circ \pi.$ However it may happens that $B$ remains very ample, as in the case of the canonical double cover $\pi\colon X\to Y$ of any Enriques surface $Y\subset \P^M$ embedded through any very ample $H$ \cite[Lemma 3.4]{gallego2008k3}. 
		
		Examples (in any dimension) in which such a $B$ is not very ample can be constructed as follows. Let $\b C\subset\P^2$ be a smooth plane curve of degree $2d+3,$ with $d\geq1,$ and take a line bundle $\vartheta$ on $\b C$ such that $2\vartheta=K_{\b C}$ and $h^0(\b C, \vartheta)=0$ (see \cite[(4.1) \& Remark 4.4]{beauville2000determinantal}). The canonical bundle is $K_{\b C}=\O_{\b C}(2d),$ hence $L=\vartheta(-d)$ is a $2$-torsion element in $\Pic(\b C).$ Then $L$ gives rise to an étale double cover $p\colon C\to \b C$ satisfying $p_\ast\O_C\cong\O_{\b C}\oplus L$ \cite[Lemma 17.2]{barth2015compact}. Setting $B'=p^\ast\O_{\b C}(d)$ and $V'=p^\ast H^0(\b C,\O_{\b C}(d)),$ by projection formula we have 
		\[
		H^0(C,B')\cong H^0(\b C,p_\ast p^\ast \O_{\b C}(d))\cong H^0(\b C,\O_{\b C}(d))\oplus H^0(\b C,\O_{\b C}(d)\o L)=H^0(\b C,\O_{\b C}(d)),
		\]
		saying that $|V'|=|B'|.$ Therefore $B'$ is non-very ample such that $\phi_{B'}=\phi_{\O_{\b C}(d)}\circ p$ is étale (onto the schematic image). This pair $(C,B')$ gives the desired example in dimension $1.$ For higher dimensions, take any smooth projective variety $Z\subset\P^N$ and set $$Y=\b C\times Z,\ H=\pi_1^\ast\left(\O_{\b C}(d)\right)\o\pi_2^\ast \left(\O_Z(1)\right),\ M=\pi_1^\ast\left(\vartheta(-d)\right)\o\pi_2^\ast\left(T\right)$$ 
		where $\pi_1\colon Y\to\b C,\pi_2\colon Y\to Z$ are the projections and $T$ is a (possibly trivial) line bundle on $Z$ such that $T^{\o2}=\O_Z.$ Clearly $H$ is very ample and 	$M$ is of $2$-torsion. Therefore we obtain a smooth projective variety $X$ and an étale double cover $\pi\colon X\to Y$ such that $\pi_\ast\O_X\cong \O_Y\oplus M$ \cite[Lemma 17.2]{barth2015compact}. Setting $B=\pi^\ast H$ and $V=\pi^\ast(H^0(Y,H)),$ then, by observing that Kunneth formula yields $$H^0(Y, H\o M)\cong H^0(Y,\pi_1^\ast(\vartheta)\o\pi_2^\ast (T(1)))\cong H^0(\b C,\vartheta)\o H^0(Z,T(1))=0,$$ we have $$H^0(X,B)\cong H^0(Y,\pi_\ast \pi^\ast H)\cong H^0(Y,H)\oplus H^0(Y,H\otimes M)=H^0(Y,H).$$ We conclude that $|V|=|B|,$ so that $B$ is ample globally generated but non-very ample and such that $\phi_B=\phi_H\circ \pi$ is étale. This is the desired example.	As a final remark, observe that these provides examples of $1$-jet spanned line bundles, i.e. separating $1$-jets at every point (see \cite[p. 273, lines 16-19]{lazarsfeld2017positivity}), which are not very ample.
	\end{rmk}
	
	To prove the main result, we will study the “separation properties” of a $B$-Ulrich bundle $\E.$ More precisely we will consider the restriction map $H^0(X,\E)\to H^0(Z,\E_{|Z}),$ where $Z$ is a $0$-dimensional closed subscheme of length $2.$ There are two possibilities: either $Z$ is contained in a $B$-line or not. Therefore it comes necessary to study also the restriction map to a $B$-line. To do this, we will use the separation lemmas in \cite[§3]{lopez2023geometrical}.

	\begin{lemma}\label{lem:restriction}
		Let $X$ be a smooth projective variety and let $B$ be a globally generated ample line bundle on $X.$ Suppose there is a subspace $V\subset H^0(X,B)$ such that $\phi=\phi_V\colon X\to\phi_V(X)=\bX\subset\P^N$ is étale. Let $\E$ be a $B$-Ulrich bundle on $X.$ Let $L\subset\bX$ be a line and let $\b L=\phi^{-1}(L)\subset X$ be the scheme-theoretic inverse image of $L.$ Then the restriction map $$H^0(X,\E)\to H^0(\b L,\E_{|\b L})$$ is surjective. In particular, the restriction map $$H^0(X,\E)\to H^0(\G,\E_{|\G})$$ is surjective for every $B$-line $\G\subset X.$
	\end{lemma}
	\begin{proof}
		Since $\b L=L\times_{\bX}X$ \cite[(4.11)]{gortz2010algebraic}, we have the cartesian diagram
		\begin{equation}\label{eq:cartesian}
			\begin{tikzcd}
				\b L\arrow[hook, rr, "\iota"]\arrow[dd, "f"']&&X\arrow[dd, "\phi"]\\
				&&\\
				L\arrow[hook, rr, "j"]&&\bX,
			\end{tikzcd}
		\end{equation} 
		where $\iota$ and $j$ are the inclusions and $f$ is the restriction of $\phi.$ The base change of an étale morphism is étale \cite[\href{https://stacks.math.columbia.edu/tag/02GO}{Tag 02GO}]{stacks-project}, hence $f$ is étale onto $L\simeq\P^1.$ Since $\P^1$ is simply connected \cite[Example IV.2.5.3]{hartshorne2013algebraic}, $\b L$ decomposes as the disjoint union $$\b L=\G_1\bigsqcup\cdots \bigsqcup \G_d,$$ where each $\G_i$ is isomorphic to $\P^1$ under $f_{|\G_i}=\phi_{|\G_i}.$ Denoting by $\iota_i$ the inclusion $\G_i\subset X$ for each $i=1,\dots,d,$ we have $j\circ f_{|\G_i}=\phi\circ\iota_i.$ In particular, we get
		\[
		B_{|\G_i}=\iota_i^\ast B\cong\iota_i^\ast\phi^\ast\O_{\b X}(1)=f_{|\G_i}^\ast j^\ast\O_{\b X}(1)=f_{|\G_i}^\ast\O_L(1).
		\]
		Then the projection formula $B\cdot \G_i=\deg(f_{|\G_i})\cdot\deg(L)=1$ says that $\G_i$ is a $B$-line.
		
		The schematic image $\bX$ is smooth \cite[Corollary 4.3.24]{liu2002algebraic}, hence the sheaf $\b\E=\phi_\ast\E$ is an Ulrich bundle for $(\bX,\O_{\bX}(1))$ by Lemma \ref{lem:claim}(vii).  As (\ref{eq:cartesian}) is cartesian and $\phi$ is affine, \cite[Proposition 12.6]{gortz2010algebraic} yields
		\begin{equation}\label{eq:flat}
			\b\E_{|L}=j^\ast\phi_\ast\E\cong f_\ast \iota^\ast\E= f_\ast(\E_{|\b L}).
		\end{equation}
		Using \cite[Lemma 3.3]{lopez2023geometrical}, we obtain the surjectivity of restriction map $H^0(\bX,\b\E)\to H^0(L,\b\E_{|L}).$  On the other hand, we have $$H^0(X,\E)\cong H^0(\bX,\phi_\ast\E)=H^0(\bX,\b\E)$$ and, thanks to (\ref{eq:flat}),  $$H^0(\b L,\E_{|\b L})\cong H^0(L,f_\ast(\E_{|\b L}))\cong H^0(L,\b\E_{|L}).$$ This proves the first part of the assertion. Moreover the space of global sections of $\E_{|\b L}$ splits as $$H^0(\b L,\E_{|\b L})=H^0(\G_1,\E_{|\G_1})\oplus\cdots\oplus H^0(\G_d,\E_{|\G_d}).$$ From this we deduce that the restriction map $$H^0(\b L,\E_{|\b L})\to H^0(\G_i,\E_{|\G_i})$$ is surjective for all $i.$ Therefore, if $\G\subset X$ is a $B$-line, the schematic image $L'=\phi(\G)\subset \bX$ is a line (Remark \ref{rmk:B-line}) and $\G$ is one of the connected components of $\phi^{-1}(L').$ By the previous part, we get the conclusion.
	\end{proof}
	
	If the $0$-dimensional closed subscheme of length $2$ is not contained in any fibre of $\phi_V,$ we will use the following.
	
	\begin{lemma}\label{lem:3.2.1}
		Let $X$ be a smooth projective variety and let $B$ be a globally generated ample line bundle on $X.$ Suppose there is a subspace $V\subset H^0(X,B)$ such that $\phi_V\colon X\to\phi_V(X)=\bX\subset\P^N$ is étale. Let $\E$ be a $B$-Ulrich bundle on $X.$ Let $Z\subset X$ be a $0$-dimensional closed subscheme of length $2$ which is not contained in any $B$-line. If $Z$ satisfies $Z\not\subset\phi_V^{-1}(\phi_V^{\hspace{0.00000001mm}}(x))$ for every $x\in Z,$ then $$r\colon H^0(X,\E)\to H^0(Z,\E_{|Z})$$ is surjective.
	\end{lemma}
	\begin{proof}
		Set $\phi=\phi_V$ and $\b Z=\phi(Z)\subset\bX.$ Observe that $\bX$ is smooth \cite[Corollary 4.3.24]{liu2002algebraic} and that the sheaf $\b\E=\phi_\ast\E$ is an Ulrich bundle for $(\bX,\O_{\bX}(1))$ by Lemma \ref{lem:claim}(vii) (see also the proof of Lemma \ref{lem:restriction}).
		Our assumption on $Z$ implies that $\b Z$ is $0$-dimensional, closed and of length $2.$ There are two cases: there is no line in $\bX$ containing $\b Z$  or there exists a line $L\subset\bX$ containing $\b Z.$
		
		Consider the first case. Then the restriction map $$H^0(\bX,\b\E)\to H^0(\b Z,\b\E_{|\b Z})$$ is surjective by \cite[Lemma 3.2]{lopez2023geometrical}. Now, let $Z'=\b Z\times_{\bX}X\subset X$ be the scheme-theoretic inverse image of $\b Z$ \cite[(4.11)]{gortz2010algebraic} and consider the cartesian diagram 
		\begin{equation}\label{eq:diagram2}
			\begin{tikzcd}
				Z'\arrow[hook, rr, "i"]\arrow[dd, "g"']&&X\arrow[dd, "\phi"]\\
				&&\\
				\b Z\arrow[hook, rr, "h"]&&\bX,
			\end{tikzcd}
		\end{equation} where $i$ and $h$ are the inclusions and $g$ is the restriction of $\phi$ to $Z'.$ Since $\phi$ is affine and (\ref{eq:diagram2}) is cartesian, we can apply \cite[Proposition 12.6]{gortz2010algebraic} to get
		\[
		\b\E_{|\b Z}=h^\ast\phi_\ast\E\cong g_\ast i^\ast\E=g_\ast(\E_{|Z'}).
		\] Since $H^0(X,\E)\cong H^0(\b X,\b\E)$ and $H^0(Z',\E_{|Z'})\cong H^0(\b Z,g_\ast(\E_{|Z'})),$ we deduce that $$H^0(X,\E)\to H^0(Z',\E_{|Z'})$$ is surjective. As $Z\subset Z'$ and $\dim Z'=0,$ also $$H^0(Z',\E_{|Z'})\to H^0(Z,\E_{|Z})$$ is surjective. Therefore $r$ is onto as required.
		
		Finally, assume that there exists a line $L\subset \bX$ passing through $\b Z.$ Let $\b L$ be the scheme-theoretic inverse image of $L.$ We know that $$\b L=\G_1\bigsqcup\cdots\bigsqcup \G_d$$ with $\G_i$ being a $B$-line (see the proof of Lemma \ref{lem:restriction}). Since $Z$ is not contained in any $B$-line, there must exist $1\leq a\neq b\leq d$ such that $x\in \G_a$ and $x'\in \G_b,$ where $Z=\{x,x'\}.$ Observe that $x$ and $x'$ cannot be infinitely near for otherwise $\supp(Z)=\{x\}$ and $\G_a\cap \G_b\neq\emptyset.$ The vector bundles $\E_{|\G_a}$ and $\E_{|\G_b}$ are globally generated, therefore $$H^0(\G_a,\E_{|\G_a})\to\E(x),\ H^0(\G_b,\E_{|\G_b})\to\E(x')$$ are surjective. It follows that the restriction map $$H^0(\b L,\E_{|\b L})\to\E(x)\oplus\E(x')=H^0(Z,\E_{|Z})$$ is surjective. The conclusion follows from the surjectivity of $H^0(X,\E)\to H^0(\b L,\E_{|\b L})$ (Lemma \ref{lem:restriction}).
	\end{proof}
	
	The next simple result of general nature will be helpful in case the $0$-dimensional closed subscheme of length $2$ is contained in a fibre of $\phi_V.$
	
	\begin{rmk}\label{rmk:restriction-fibre}
		Let $f\colon X\to Y$ be an affine morphism of schemes and let $\F$ be a coherent sheaf on $X$ such that $f_\ast\F$ is generated by global sections. Suppose the schematic fibre $X_y=f^{-1}(y)$ over $y\in Y$ is non-empty. Then the restriction map $$H^0(X,\F)\to H^0(X_y,\F_{|X_y})$$ is surjective.
		
		Indeed, consider the cartesian diagram 
		\[
		\begin{tikzcd}
			X_y=X\times_Y\Spec\left(\C(y)\right)\arrow[hook, rr, "i"]\arrow[dd, "g"']&&X\arrow[dd, "f"]\\
			&&\\
			\{y\}=\Spec\left(\C(y)\right)\arrow[hook, rr, "j"]&&Y
		\end{tikzcd}
		\] with $i$ and $j$ being the inclusions and $g$ the base change of $f.$ Since $f_\ast \F$ is globally generated on $Y,$ the map $$H^0(Y,f_\ast\F)\to (f_\ast\F)(y)=H^0(\{y\},(f_\ast\F)_{|\{y\}})$$  is surjective. On the other hand, \cite[Proposition 12.6]{gortz2010algebraic} gives the isomorphism
		\[
		(f_\ast\F)_{|\{y\}}=j^\ast f_\ast\F\cong g_\ast i^\ast\F=g_\ast(\F_{|X_y}).
		\] As $H^0(X,\F)\cong H^0(Y,f_\ast\F)$ and $H^0(X_y,\F_{|X_y})\cong H^0(\{y\},g_\ast(\F_{|X_y})),$ the assertion follows. 
	\end{rmk}
	
	\begin{lemma}\label{lem:3.2.2}
		Let $X$ be a smooth projective variety and let $B$ be a globally generated ample line bundle on $X.$ Suppose there is a subspace $V\subset H^0(X,B)$ such that $\phi_V\colon X\to\phi_V(X)=\bX\subset\P^N$ is étale. Let $\E$ be a $B$-Ulrich bundle	on $X.$ Let $Z\subset X$ be a $0$-dimensional closed subscheme of length $2$ which is not contained in any $B$-line. Then $$r\colon H^0(X,\E)\to H^0(Z,\E_{|Z})$$ is surjective.
	\end{lemma}
	\begin{proof}
		Write $\phi=\phi_V$ and denote by $F_y$ the schematic fibre $\phi^{-1}(\phi(y))$ for every $y\in X.$ If $Z\not\subset F_x$ for every $x\in Z,$ the claim follows by Lemma \ref{lem:3.2.1}. On the other hand, since $\phi_\ast\E$ is Ulrich on $\bX$ (see the proof of Lemma \ref{lem:restriction}), hence globally generated, we know from Remark \ref{rmk:restriction-fibre} that  $$r_{y}\colon H^0(X,\E)\to H^0(F_y,\E_{|F_y})$$ is surjective for every $y\in X.$  In case $Z\subset F_x$ for $x\in Z,$ the schemes $F_x$ and $Z$ consist of distinct points since $\phi$ is unramified. Therefore $$r'\colon H^0(F_x,\E_{|F_x})\to H^0(Z,\E_{|Z})$$ is surjective. Since $r=r'\circ r_x,$ the assertion follows.
	\end{proof}
	
	\begin{lemma}\label{lem:3.2}
		Let $X$ be a smooth projective variety and let $B$ be a globally generated ample line bundle on $X.$ Suppose there is a subspace $V\subset H^0(X,B)$ such that $\phi_V\colon X\to\phi_V(X)=\bX\subset\P^N$ is étale. Let $\E$ be a $B$-Ulrich bundle on $X.$ Let $Z\subset X$ be a $0$-dimensional closed subscheme of length $2$ and let $r\colon H^0(X,\E)\to H^0(Z,\E_{|Z})$ be the restriction map. Then $r$ is surjective if one of the following holds:
		\begin{itemize}
			\item[(i)] There is no $B$-line containing $Z.$
			\item[(ii)] There is a $B$-line $\G\subset X$ containing $Z$ and $\E_{|\G}$ is ample on $\G.$
		\end{itemize} 
	\end{lemma}
	\begin{proof}
		In case (i), the claim follows by Lemma \ref{lem:3.2.2}. Then suppose that (ii) holds.	 Lemma \ref{lem:restriction} says that $r_\G\colon H^0(X,\E)\to H^0(\G,\E_{|\G})$ is surjective, and the (very) ampleness of $\E_{|\G}$ on $\G\simeq\P^1$ (Remark \ref{rmk:B-line}) tells that $r_Z\colon H^0(\G,\E_{|\G})\to H^0(Z,\E_{|Z})$ is surjective. The map $r$ factors as the composition $r=r_Z\circ r_\G,$ hence the assertion follows.
	\end{proof}
	
	Now we relate the separation properties of the $B$-Ulrich bundle $\E$ with the Seshadri constant of the tautological line bundle $\O_{\P(\E)}(1).$
	
	\begin{lemma}\label{lem:step3-4'}
		Let $X$ be a smooth projective variety and let $B$ be a globally generated ample line bundle on $X.$ Suppose there is a subspace $V\subset H^0(X,B)$ such that $\phi=\phi_V\colon X\to\phi_V(X)=\bX\subset\P^N$ is étale. Let $\E$ be a $B$-Ulrich bundle on $X.$ Let $x\in X$ be a point and suppose one of the following conditions holds:
		\begin{itemize}
			\item[(i)] There are no $B$-lines passing through $x.$
			\item[(ii)] $\E_{|\G}$ is ample on every $B$-line $\G\subset X$ passing through $x.$
		\end{itemize}
		Then $\eps(\O_{\P(\E)}(1);y)\geq1$ for all $y\in\P(\E(x)).$ 
	\end{lemma}
	\begin{proof}
		Let $\pi\colon\P(\E)\to X$ be the natural projection, and write $\xi=\O_{\P(\E)}(1)$ and $\xi_x:=\xi_{|\P(\E(x))}\simeq\O_{\P^{\rk(\E)-1}}(1).$ Fix a point $y\in\P(\E(x)).$ According to Remark  \ref{rmk:seshadri-jet}, it suffices to show that  the restriction map $$r_Z\colon H^0(\P(\E),\xi)\to H^0(Z,\xi_{|Z})$$ to every $0$-dimensional closed subscheme $Z\subset \P(\E)$ of length $2$ with $\supp(Z)=\left\{y\right\}$ is surjective. If $Z\subset\P(\E(x)),$ since $\xi_{|Z}=(\xi_x)_{|Z},$ one has the commutative diagram
		\begin{equation}\label{eq:diagram-projective}
			\begin{tikzcd}
				H^0(\P(\E),\xi)\arrow[rr, "r_Z"]\arrow[dd, "\rho"] &&  H^0(Z,\xi_{|Z})\arrow[dd, "\cong"]\\ 
				&&\\
				H^0(\P(\E(x)),\xi_x)\arrow[rr, "r_{Z,x}"]&& H^0(Z,(\xi_x)_{|Z}).
			\end{tikzcd}
		\end{equation} The map $\rho$ is surjective \cite[Remark 2.3]{lopez2023geometrical}, as well as $r_{Z,x}$ since $\xi_x$ is very ample on $\P(\E(x)).$ Therefore $r_Z$ must be surjective by the commutativity of (\ref{eq:diagram-projective}). Suppose $Z \not\subset\P(\E(x)),$ so that $\pi(Z)\subset X$ is a $0$-dimensional closed subscheme of length $2.$ Thanks to our assumptions, the restriction map $H^0(X,\E)\to H^0(\pi(Z),\E_{|\pi(Z)})$ is surjective by Lemma \ref{lem:3.2}: if there is no $B$-line containing $\pi(Z),$  we are in case (i); if $\pi(Z)\subset\G$ for a $B$-line $\G\subset X,$ then we are in the case (ii). It follows that $r_Z$ is surjective as claimed.
	\end{proof}
	
	Now we are ready to prove Theorem \ref{thm:ulrich-augmented} and its corollary.
	
	\begin{proof}[Proof of Theorem \ref{thm:ulrich-augmented}] The last part of the assertion is a consequence of the characterization of the augmented base locus in Remark \ref{rmk:augmented-vector}. Hence we only need to show the first one.
		The “if part” follows by Lemma \ref{lem:step1} and Remark \ref{rmk:B-line}. Conversely, suppose that $\eps(\E;x)=0.$ By Remark \ref{rmk:augmented} we know that there is a point $z\in \P(\E(x))$ such that $z\in\B_+(\O_{\P(\E)}(1)),$ hence such that $\eps(\O_{\P(\E)}(1);z)=0$ (Remark \ref{rmk:augmented-seshadri}). Let $\mc{B}_x$ be the set of all $B$-lines passing through $x.$ Then $\mc{B}_x\neq\emptyset,$ for otherwise Lemma \ref{lem:step3-4'}(i) would imply $\eps(\O_{\P(\E)}(1);z)\geq1.$ Moreover,  if $\E_{|\G_x}$ is ample on every $\G_x\in\mc{B}_x,$ then $\eps(\O_{\P(\E)}(1);z)\geq1$ by Lemma \ref{lem:step3-4'}(ii), giving a contradiction. Therefore we conclude that there must exist a $B$-line $\G\subset X$ such that $\E_{|\G}$ is not ample.
	\end{proof}

	\begin{proof}[Proof of Corollary \ref{cor:ulrich-augmented}]
		This is a simple consequence of Theorems \ref{thm:V-big} - \ref{thm:ulrich-augmented} and of Remark \ref{rmk:augmented-vector}.
	\end{proof}
	
	Let's see an easy example where the augmented base locus can be computed explicitly.
	
	\begin{ex}\label{ex:DP-surface-augmented}
		Consider a Del Pezzo surface $(S,-K_S)$ of degree $3\le d\le 7.$ Then $\phi_{-K_S}\colon S\hookrightarrow\P^{9-d}$ is an embedding and, as is well known, there are finitely many lines lying on $S.$ In our situation at least two of them are disjoint, say $\l_1$ and $\l_2.$ Then $\E=(\l_1-\l_2)(-K_S)$ is a $(-K_S)$-Ulrich line bundle on $S$ \cite[Proposition 4.1(i)]{beauville2018introduction}. Since $\E^2=d-2>0,$ this is (nef and) big. Hence $\B_+(\E)\neq S$ (Remark \ref{rmk:V-big}).  We claim that $$\B_+(\E)=\ \l_1\ \bigsqcup\ \left(\ \bigcup_{\substack{\l\subset S\ \mathrm{line}\colon\\
				\l\cap\l_1=\emptyset,\ \l\neq\l_2,\ \l\cap\l_2\neq\emptyset}}\l\ \right).$$ 
		It particular, all these $\E$'s are big Ulrich bundles which are not ample (Corollary \ref{cor:ulrich-augmented}).
		
		To see this, by Theorem \ref{thm:ulrich-augmented} we only need to find the lines $\l\subset S$ such that $\E\cdot\l=0.$ As all lines in $S$ are $(-1)$-curves and	$\E\cdot\l=\l_1\cdot\l-\l_2\cdot\l+1,$ the claim immediately follows.
	\end{ex}
	
	Unlike what happens for the ampleness and ($1$-)very ampleness (see Theorem \ref{thm:ulrich-ampleness}), bigness and V-bigness for a $B$-Ulrich bundle, with $B$ as above, are not equivalent.
	
	\begin{rmk}
		Let $Q_n\subset\P^{n+1}$ be a smooth quadric of dimension $n\ge 7, n\neq 10.$ Then the spinor bundles $\mc{S},\mc{S'},$ if $n$ is even, and $\mc{S''},$ when $n$ is odd, are the only indecomposable Ulrich bundles on $Q_n$ \cite[Proposition 2.5]{beauville2018introduction} and are big in this situation \cite[Theorem 1]{lopez2023non}. However, their restriction to every line in $Q_n$ is not ample \cite[Corollary 1.6]{ottaviani1988spinor}. Since $Q_n$ is covered by lines, by Corollary \ref{cor:ulrich-augmented}(a) we conclude that $\mc{S},\mc{S'}$ and $\mc{S''}$ are not V-big.
	\end{rmk}
	
	Thanks to Corollary \ref{cor:ulrich-augmented} and to Theorem \ref{thm:ulrich-ampleness} we can find new examples of (V-)big Ulrich bundles which are not ample (and whose $c_1$ is not very ample). See also \cite[Remark 4.3]{lopez2023geometrical}.
	
	\begin{ex}
		Let $S\subset\P^3$ be the cubic surface (with hyperplane section $\O_S(1)=-K_S$). Since $S$ is not covered by lines, all Ulrich bundles on $S$ are trivially V-big by Corollary \ref{cor:ulrich-augmented}. Now, looking at $S$ as the blow-up of the plane at $6$ points in general position, denote by $\w H$ the pullback of a line and by $E_i$ the exceptional divisors, for $i=1,\dots,6.$ By \cite[Examples 3.6 - 4.7]{casanellas2012stable} there exist three Ulrich bundles $\E_a,\E_b,\E_d$ of rank $2$ and an Ulrich bundle $\E_h$ of rank $3$ on $S$ such that 
		\[
		c_1(\E_a)=2\w H,\ \ \ c_1(\E_b)=3\w H-E_1-E_2-E_3,\ \ \ c_1(\E_d)=4\w H-2E_1-\sum_{i=2}^5E_i, \ \ \ c_1(\E_h)=6\w H-2\sum_{i=1}^4 E_i-E_5.
		\] Each of the above $c_1$ is clearly not ample, for instance $c_1\cdot E_6=0$ for all of them. Then $\E_a,\E_b,\E_d,\E_h$ cannot be ample, for otherwise they would be $1$-very ample (Theorem \ref{thm:ulrich-ampleness}) and their first Chern class would be so as well (Remark \ref{rmk:k-very}).
	\end{ex}
	
	We can find big non-ample Ulrich bundles even if its Chern class is very ample and the variety is covered by lines.
	
	\begin{ex}\label{ex:DP-threefold-V-big}
		Let $(X,B)$ be a Del Pezzo threefold of degree $3\leq d\leq 5.$ Then $X$ is covered by lines \cite[Proposition III.1.4(ii)]{iskovskikh1980anticanonical}, and in fact the Fano scheme of lines $F(X)$ is a smooth irreducible surface \cite[Propositions III.1.3(iii) - III.1.6(i) \& Remark III.1.5]{iskovskikh1980anticanonical}. Let $\E$ be any stable special Ulrich bundle of rank $2$ on $X$ \cite[Theorem 1.1]{ciliberto2023ulrich}. In particular $c_1(\E)=2B$ is very ample. As $\E(-1)$ is still stable and $c_1(\E(-1))=0,H^1(X,\E(-2))=0,$ 	we see that $\E(-1)$ is an instanton bundle in the sense of \cite[Definition 1.1]{kuznetsov2012instanton}. Since $\E_{|L}\cong\O_L(1)^{\oplus2}$ for a general line $[L]\in F(X)$ \cite[(Proof of) Lemma 4.8(i)]{ciliberto2023ulrich}, by \cite[Theorem 3.17]{kuznetsov2012instanton} there exists a non-zero divisor $D_\E\subset F(X)$ which parameterizes the lines $L\subset X$ such that $\E(-1)_{|L}\cong \O_L(-1)\oplus\O_L(1).$ Then Corollary \ref{cor:ulrich-augmented} and (\ref{eq:ulrich-augmented}) tell that $\E$ is not ample with non-empty augmented base locus  $$\B_{+}(\E)=\bigcup_{[L]\in D_\E}L.$$ Consider the incidence correspondence 
		\[
		\begin{tikzcd}
			&\mc{S}=\left\{(x,[\G])\in X\times F(X)\ |\ x\in\G \right\}\subset X\times F(X)\arrow[from=1-2, to=2-1, "\pi_1"']\arrow[from=1-2, to=2-3, "\pi_2"]\\
			X && F(X)
		\end{tikzcd}
		\]
		with projections $\pi_1,\pi_2.$ As every fibre $\pi_2^{-1}([L])\cong L$ is smooth irreducible of dimension $1,$ a dimension count shows that $\pi_1(\pi_2^{-1}(D_\E))\subsetneq X$ is a proper subset. As $\B_+(\E)=\pi_1(\pi_2^{-1}(D_\E)),$ we conclude that $\E$ is V-big.
	\end{ex}
	
	We observe that Theorem \ref{thm:ulrich-augmented} and Corollary \ref{cor:ulrich-augmented} no longer hold if we do not assume $\phi_V$ to be étale, in particular if we allow the ramification locus to be non-empty.
	
	\begin{rmk}\label{rmk:DP-surface}
		The Del Pezzo surface $(S,-K_S)$ of degree $d=2,$ which is a finite flat double cover $\phi_{-K_S}\colon S\to\P^2$ branched over a smooth plane quartic curve, supports a non-big $(-K_S)$-Ulrich line bundle $\E$ (see Proposition \ref{prop:ulrich-DP-surface} and Example \ref{ex:DP-surface}). This means that $\B_+(\E)=S.$ 	On the other hand, it is well known that $S$ is not covered by $(-K_S)$-lines (see also Example \ref{ex:DP-surface-augmented}). Therefore the equality in (\ref{eq:ulrich-augmented}) cannot hold for $\E.$
	\end{rmk}
	
	The previous example was easy because the variety contains a finite number of $B$-lines. However the same problem may arise also when the variety is covered by $B$-lines.
	
	\begin{rmk}\label{rmk:DP-threefold}
		Let $(X,B)$ be a Del Pezzo $3$-fold of degree $d=2$ and let $\E$ be the stable special $B$-Ulrich bundle of rank $2$ provided by Proposition \ref{prop:ulrich-DP-threefold}. As above, we are not in the situation of Theorem \ref{thm:ulrich-augmented} since the morphism $\phi_B\colon X\to\P^3$ is a finite flat double cover branched over a smooth quartic surface. But, on the contrary of the previous remark, $X$ is covered by $B$-lines (see Remark \ref{rmk:seshadri-B-line}). More precisely the Fano scheme of ($B$-)lines $F(X)$ is an irreducible surface \cite[Remark III.1.7]{iskovskikh1980anticanonical}. In Example \ref{ex:DP-threefold} we showed that $\E$ is non-big. Therefore $\B_+(\E)=X.$ However, in \cite[Proof of Theorem D, Step 1]{faenzi2014even} it is proved that $\E_{|\G}\cong \O_{\G}(1)^{\oplus2}$ for a general $B$-line $[\G]\in F(X).$ Exactly as in Example \ref{ex:DP-threefold-V-big}, $\E(-B)$ is an instanton bundle on $X$ \cite[Definition 1.1]{kuznetsov2012instanton}. Therefore the $B$-lines $\G\subset X$ such that $\E_{|\G}\cong\O_\G^{\oplus2}$ form a non-zero divisor $D_\E\subset F(X)$ by \cite[Theorem 3.17]{kuznetsov2012instanton}. Consider the incidence correspondence
		\[
		\begin{tikzcd}
			&\mc{S}=\left\{(x,[\G])\in X\times F(X)\ |\ x\in\G \right\}\subset X\times F(X)\arrow[from=1-2, to=2-1, "\pi_1"']\arrow[from=1-2, to=2-3, "\pi_2"]\\
			X && F(X),
		\end{tikzcd}
		\] with projections $\pi_1,\pi_2.$ Every fibre $\pi_2^{-1}([\G])\cong \G$ is smooth irreducible of dimension $1.$ For dimensional reasons, $\pi_1(\pi_2^{-1}(D_\E))\subsetneq X$ must be a proper subset. By construction, every point $x\in X\backslash \pi_1(\pi_2^{-1}(D_\E))$ is crossed only by $B$-lines $\G\subset X$ such that $\E_{|\G}\cong \O_\G(1)^{\oplus2}.$ We conclude that the union over all $B$-lines $\G\subset X$ such that $\E_{|\G}$ is not ample cannot coincide with $X.$ This says that Theorem \ref{thm:ulrich-augmented} and Corollary \ref{cor:ulrich-augmented} do not apply even in this situation.
	\end{rmk}
	
	We point out that the assumptions on $B$ considered so far are not necessary: there are examples in which $\phi_B$ is not étale but the characterization of Theorem \ref{thm:ulrich-augmented} holds.
	
	\begin{rmk}\label{rmk:ulrich-curve}
		First we observe that given any smooth projective curve $C$ of genus $g\geq0$ with a globally generated ample line bundle $B,$ one can easily see that $\E_{L}=L(B)$ is $B$-Ulrich for every line bundle $L\in \Pic^{g-1}(C)\backslash\Theta,$ with $\Theta$ being the divisor which contains all effective line bundles on $C$ of degree $g-1.$ Note that $\deg(\E_L)=\deg B+g-1.$
		
		Any pair $(C,B)\neq(\P^1,\O_{\P^1}(1))$ is not a $B$-line and, in this case, all $\E_L$'s are ample since $\deg(\E_L)>0.$ Therefore $\B_+(\E_L)=\emptyset$ accordingly with the fact that $C$ contains no $B$-lines.
		
		However, taking for instance $(C,B)$ with $g=1$ and $\deg B=2,$ we see that $\E_L$ is ample but not ($1$-)very ample. This shows that also Theorem \ref{thm:ulrich-ampleness} does not hold when $\phi_B$ is ramified: in fact $\phi_B\colon C\to\P^1$ is a finite flat double cover with $2$ branch points. Anyway, if $\deg B\geq g+2,$ the $B$-Ulrich line bundle $\E_L$ is ($1$-)very ample even if $\phi_V$ is not étale for all $|V|\subset|B|.$ 
	\end{rmk}
	
	Finally we prove Theorem \ref{thm:ulrich-ampleness}, which is nothing but the analogue of \cite[Theorem 1]{lopez2023geometrical} in this setting.
	
	\begin{proof}[Proof of Theorem \ref{thm:ulrich-ampleness}]
		It's clear that (2) implies (3), while the fact that (1) implies (3) is given by 	Remark \ref{rmk:k-very}. Corollary \ref{cor:ulrich-augmented}(b) gives the equivalence of (3) and (4). Let $\pi\colon\P(\E)\to X$ be the natural projection and let $\xi=\O_{\P(\E)}(1)$ be the tautological line bundle on $\P(\E).$ Now assume (4) and let $Z\subset \P(\E)$ (resp. $Z'\subset X$) be a $0$-dimensional closed subscheme of length $2.$ To prove (2) (resp. (1)), we show that 
		\[
		r_Z\colon H^0(\P(\E),\xi)\to H^0(Z,\xi_{|Z}) \hspace{1cm} (\mathrm{resp.}\ r'\colon H^0(X,\E)\to H^0(Z',\E_{|Z'}))
		\] 
		is surjective. If $Z\subset\P(\E(x))$ for some $x\in X,$ we obtain the same commutative diagram of (\ref{eq:diagram-projective}), which then forces $r_Z$ to be surjective. Suppose $Z\not\subset\P(\E(x))$ for every $x\in X.$ Then $\pi(Z)\subset X$ is a $0$-dimensional closed subscheme of length $2.$ Let $r\colon H^0(X,\E)\to H^0(\pi(Z),\E_{|\pi(Z)})$ be the restriction map. If there is no $B$-line passing through $\pi(Z)$ (resp. $Z'$), which happens in particular when $X$ contains no $B$-lines, then $r$ (resp. $r'$) is surjective by Lemma \ref{lem:3.2}(i), hence so is $r_Z.$ Analogously, if $\G\subset X$ is a $B$-line containing $\pi(Z)$ (resp. $Z'$), then we know that $r$ (resp. $r'$) is surjective thanks to Lemma \ref{lem:3.2}(ii). Therefore $r_Z$ is surjective as desired. This concludes the proof.		
	\end{proof}

\end{document}